\documentclass{siamart171218}
\usepackage[utf8]{inputenc}
\usepackage{amssymb}
\usepackage{enumitem}
\usepackage{graphicx}
\usepackage{subcaption}
\usepackage{graphicx}
\usepackage{mathrsfs}
\usepackage{amsfonts}
\usepackage{graphicx}
\usepackage{graphicx}
\usepackage{subcaption}
\usepackage{amsmath}
\usepackage{xcolor}
 \usepackage{epstopdf}
 \usepackage{boondox-cal}
\usepackage{txfonts}
\usepackage{dsfont}
\usepackage{txfonts}
\usepackage{amssymb}
\usepackage{amsmath}
 \usepackage{epstopdf}
\usepackage{amsmath,amssymb}
\usepackage{mathrsfs,color}

\theoremstyle{remark}
\newtheorem{remark}{Remark}  
\theoremstyle{example}
\newtheorem*{example}{Example}  

\def\theequation{\arabic{section}.\arabic{equation}}
\numberwithin{equation}{section}

\newcommand{\TheTitle}{A Relaxed Step-Ratio Constraint for Time-Fractional Cahn--Hilliard Equations: Analysis and Computation}

\title{{\TheTitle}\thanks{The work was partially supported by
the National Natural Science Foundation of China
 (Nos.12461069, 11961057), the Science and Technology Project
of Guangxi (No. GuikeAD21220114), and
the Natural Science Foundation of Guangdong Province
of China (No. 2025A1515012121).}}
\author{
Shipeng Li\thanks{School of Mathematics and Statistics,
Guangxi Normal University, Guilin 541006, China.}
\and
Hengfei Ding\thanks{1.School of Mathematics and Statistics,
 Guangxi Normal University, Guilin 541006, China;
2.The Center for Applied
Mathematics of Guangxi (GXNU), Guilin 541006, China;
3.Guangxi Colleges and Universities Key Laboratory of
Mathematical Model and Application (GXNU),
 Guilin 514006, China. (E-mail:dinghf05@163.com).}}

\begin{document}
\maketitle
\begin{abstract}
Numerical solutions of time-fractional differential equations encounter significant challenges arising from solution singularities at the initial time. To address this issue, the construction of nonuniform temporal meshes satisfying $\tau_k/\tau_{k-1} \geq 1$ has emerged as an effective strategy, where $\tau_k$ represents the $k$-th time-step size. For the time-fractional Cahn-Hilliard equation, Liao et al.~[\textit{IMA J. Numer. Anal.}, \textbf{45} (2025), 1425--1454] developed an analytical framework using a variable-step L2 formula with the constraint $0.3960 \leq \tau_k/\tau_{k-1} \leq r^*(\alpha)$, where $r^*(\alpha) \geq 4.660$ for $\alpha \in (0,1)$.
The present work makes substantial theoretical progress by introducing innovative splitting techniques that relax the step-size ratio restriction to $\tau_k/\tau_{k-1} \leq \rho^*(\alpha)$, with $\rho^*(\alpha) > \overline{\rho} \approx 4.7476114$. This advancement provides significantly greater flexibility in time-step selection. Building on this theoretical foundation, we propose a refined L2-type temporal approximation coupled with a fourth-order compact difference spatial discretization, yielding an efficient numerical scheme for the time-fractional Cahn-Hilliard problem. Our rigorous analysis establishes the scheme's fundamental properties, including unique solvability, exact discrete volume conservation, proper energy dissipation laws, and optimal convergence rates.
For practical implementation, we construct a specialized nonuniform mesh that automatically satisfies the relaxed constraint $\rho^*(\alpha) > 4.7476114$. Comprehensive numerical experiments validate the method's robustness and effectiveness, confirming theoretical predictions regarding accuracy, stability, and energy dissipation across various test cases. These results collectively demonstrate the method's superior performance in challenging scenarios, establishing its practical utility for solving time-fractional partial differential equations.
\end{abstract}


\section{Introduction}\label{sec: introduction}

The Cahn-Hilliard equation was originally introduced by Cahn and Hilliard to model the intricate processes of phase separation and coarsening in binary systems within thermodynamic frameworks \cite{Cahn}. Over time, this equation has not only established itself as a fundamental tool in engineering and materials science \cite{Bourdin, Chen, Du2004, Du2005}, but has also found widespread applications across diverse scientific disciplines \cite{Bertozzi2006, Bertozzi2007, Shen1, Shen2}. In recent developments, researchers have recognized that the time-fractional Cahn-Hilliard equation serves as a natural generalization of the classical version. This extended formulation proves particularly valuable for systems whose dynamics depend on historical states \cite{Sadaf, Zhao2019, Fritz}. The incorporation of fractional time derivatives enables the description of memory effects and anomalous diffusion during phase separation, features that remain beyond the reach of conventional Cahn-Hilliard models.

Here, we study the following time-fractional Cahn-Hilliard (TFCH) equation
\begin{align}\label{eq.1.1}
\left\{
\begin{aligned}
  &\partial^\alpha_t u(x,t) = \kappa \Delta f(u)+\kappa \varepsilon^2 \Delta v(x,t),  \quad x\in\Omega,\; t\in(0,T], \\
  & v(x,t)=-\Delta u(x,t),  \quad x\in\Omega,\; t\in(0,T], \\
 &u(x,0) =u_0 (x), \quad x\in\overline{\Omega},
\end{aligned}
\right.
\end{align}
where the spatial domain is given by $\Omega = (a,b)$, and $T > 0$ represents the final time. The nonlinear term is defined as $f(u) = u^3 - u$, while $\varepsilon > 0$ and $\kappa > 0$ denote the interface width parameter and mobility coefficient, respectively. The Caputo fractional derivative $\partial^\alpha_t$ with $0 < \alpha < 1$ is expressed as
\begin{equation}\label{eq.1.2}
\partial^\alpha_t y(t) = I_t^{1-\alpha }y' (t) = \frac{1}{\Gamma(1-\alpha)}\int_0^t (t-s)^{-\alpha} y'(s) \mathrm{d}s,
\end{equation}
where $I_t^\mu$ stands for the Riemann-Liouville fractional integral of order $\mu > 0$, given by
\begin{align*}
I_t^\mu y (t) = \frac{1}{\Gamma(\mu)}\int_0^t (t-s)^{\mu-1} y(s) \mathrm{d}s.
\end{align*}

It is well established that the classical Cahn-Hilliard (CH) model \cite{Elliott} preserves mass conservation, which can be written as
\begin{equation}\label{eq.1.4}
\int_\Omega u(x,t) \mathrm{d}x = \int_\Omega u(x,0) \mathrm{d}x, \quad t \in (0,T].
\end{equation}
Moreover, the model satisfies an energy dissipation property
\begin{align*}
\frac{ \mathrm{d}E}{\mathrm{d}t}=-\kappa \|\nabla f(u)+ \varepsilon^2 \nabla v\| \leq 0 \;  \quad t \in (0,T],
\end{align*}
where $E[u](t)$ represents the Ginzburg–Landau energy functional defined by
\begin{equation}\label{eq.1.6}
E[u](t) = \int_\Omega \left( \frac{\varepsilon^2}{2}|\nabla u|^2 + F(u) \right) \mathrm{d}x,
\end{equation}
with the double-well potential $F(u) = \frac{1}{4}(1 - u^2)^2$. When the fractional order $\alpha$ approaches $1^-$, the TFCH model reduces to the standard integer-order CH model. A key question is whether the TFCH model retains these two fundamental properties.
Previous studies have addressed this question from different perspectives. Tang et al. \cite{Tang} established a global energy dissipation law, proving that $E[u](t)\leq E[u](0)$ for all $t > 0$. Quan et al. \cite{Quan1} introduced nonlocal and weighted energy dissipation laws for time-fractional phase-field models. In \cite{Liao2021}, Liao et al. derived an equivalent formulation involving the Riemann–Liouville derivative and demonstrated that the energy dissipation laws of certain numerical schemes, including the L1, L$1^+$, and L$1_R$ methods, are asymptotically compatible with their integer-order counterparts. Further analysis of these schemes can be found in \cite{Ji,Liao1,Liao2}. More recently, Quan et al. \cite{Quan2} constructed a modified energy functional $\widetilde{E}_\alpha[u]$, which serves as a decreasing upper bound for the original energy $E[u]$. As shown in Proposition 5.1 of \cite{Quan2}, this modified energy is asymptotically compatible with the classical Ginzburg–Landau energy functional.

Due to the nonlocal nature of fractional derivatives and the presence of nonlinear terms in equation (\ref{eq.1.1}), deriving analytical solutions for such problems is highly challenging. As a result, numerical methods have become an essential tool for investigating the dynamic behavior of these equations. A crucial aspect of such numerical approaches lies in the accurate approximation of the Caputo derivative, which requires the development of efficient numerical differentiation formulas.
Over the past few decades, significant progress has been made in constructing numerical approximations for the Caputo derivative of order $\alpha \in (0,1)$. A straightforward and widely used approach, proposed in \cite{Oldham}, employs piecewise linear interpolation to derive a weighted numerical integral, leading to the well-known L1 formula. Wu and Sun \cite{Sun} later demonstrated that the L1 formula achieves an accuracy of order $2 - \alpha$. To further improve convergence orders, Gao et al. \cite{Gao2014} introduced the L$1-2$ formula on a uniform mesh, which attains an order of $3 - \alpha$ at interior points $t_k$ ($k \geq 2$), though the accuracy reduces to $2 - \alpha$ at the first time point $t_1$.
Building upon these developments, Alikhanov \cite{Alikhanov} proposed a more refined scheme known as the L$2-1_\sigma$ formula, which achieves a uniform convergence order of $3 - \alpha$ at the shifted grid points $t_{k+\sigma}$, where $\sigma = 1 - {\alpha}/{2}$. Independently, Lv and Xu \cite{Lv} constructed a different approximation, called the L2 formula, based on piecewise quadratic interpolation. While this method shares the same convergence order as the L$2-1_\sigma$ formula, it employs a distinct interpolation strategy.
Building upon the aforementioned numerical differential formulas, significant progress has been made in developing numerical solutions for time-fractional differential equations. A common framework in these studies relies on the key assumption that the analytical solution and its higher-order derivatives satisfy prescribed smoothness conditions. Such regularity requirements are essential for ensuring the stability and convergence of numerical schemes.

Recent studies have revealed that exact solutions of time-fractional differential equations exhibit weak singularities at $t=0$. This weak singularity in the solution $u$ poses a significant challenge, as it prevents numerical schemes from achieving their expected convergence orders. To address this issue, researchers have proposed constructing numerical differentiation formulas on non-uniform meshes \cite{Kopteva,Liao2020,Stynes2017}. The underlying principle of this approach involves concentrating grid points more densely near the initial point $t=0$ while maintaining sparser distributions away from the origin. This temporal adaptation helps better capture the singular behavior near the initial time.
Following this concept, substantial progress has been made in developing numerical differentiation formulas for Caputo derivatives on non-uniform meshes. Quan et al. \cite{Quan3} pioneered the investigation of $H^1$-stability for L2-type formulas on general nonuniform meshes when applied to subdiffusion equations. Their work established the positive semidefiniteness with the difference multiplier $\triangledown_\tau w^n$ under specific time step ratio constraints, requiring $0.4573328 \leq \tau_{k}/\tau_{k-1} \leq 3.5615528$ for all $k \geq 2$. Subsequently, Liao et al. \cite{Liao2024} made significant advances in analyzing the variable-step fractional order BDF2 implicit scheme for the time-fractional Cahn-Hilliard model. Their work derived a local discrete energy dissipation law and substantially relaxed the time step ratio condition to $0.3960 \leq \tau_{k}/\tau_{k-1} \leq r^*(\alpha)$ for $k \geq 2$, where $r^*(\alpha) \geq 4.660$ when $\alpha \in (0,1)$.

Building upon these theoretical advances, this work develops a numerical discretization of the Caputo derivative using the L2 formula on general nonuniform meshes. Our approach focuses on relaxing the temporal step ratio constraints to achieve greater computational flexibility. Furthermore, we construct an optimized nonuniform mesh that satisfies the proposed step ratio condition, leading to an enhanced numerical scheme for solving equation (\ref{eq.1.1}).
The main theoretical and computational contributions of this study include the following key aspects:
\begin{enumerate}
\item {\bf Improved time step ratio condition for the discrete Caputo derivative:} Through careful modification of the splitting components, the derived time step ratio condition becomes $1 \leq \tau_{k}/\tau_{k-1} \leq \rho^*(\alpha)$ for $k \geq 2$, where the upper bound satisfies $\rho^*(\alpha) > \overline{\rho} \approx 4.7476114$. This new condition offers superior adaptability to non-uniform meshes compared to existing alternatives. A comprehensive analysis is presented in Section \ref{sec: The discrete Caputo derivative}.
\item {\bf Development of an effective numerical scheme:} By combining the L2 formula with a fourth-order compact difference approximation for spatial derivatives, an efficient numerical scheme for equation (\ref{eq.1.1}) is naturally derived.

\item {\bf Rigorous theoretical analysis:} The study provides complete theoretical justification for the proposed difference scheme, including proofs of unique solvability, discrete volume conservation properties, energy decay law, and convergence behavior.

\item {\bf Construction of an optimized non-uniform mesh:} A novel non-uniform mesh is developed that better satisfies the improved time step ratio condition compared to existing mesh constructions.
\end{enumerate}

The remainder of this paper is organized as follows.
Section \ref{sec: The discrete Caputo derivative} introduces
 the L2 nonuniform scheme for temporal discretization.
 The complete numerical scheme is developed in
 Section \ref{sec: Establishment of the numerical scheme},
 followed by rigorous theoretical analysis in Section
\ref{sec:theoretical-analysis}.
Section \ref{sec: Numerical examples} presents extensive
numerical experiments validating the theoretical results.
Finally, Section \ref{sec: Concluding remarks} provides
concluding remarks and discusses potential extensions of this work.

\section{The novel decomposition formulation of the L2 scheme for discretizing the Caputo derivative}
\label{sec: The discrete Caputo derivative}
For a positive integer $N$ and given time $T$, the interval $[0, T]$ is divided into $0 = t_0 < t_1 < \cdots < t_k < t_{k+1} < \cdots < t_N = T$, with variable stepsizes $\tau_k = t_k - t_{k-1}$ for $1 \leq k \leq N$. The maximum step size is denoted by $\tau_{\max} := \max_{1 \leq k \leq N} \tau_k$.
The adjacent time-step ratios are defined as $\rho_1 := 0$ and $\rho_k := \tau_k / \tau_{k-1} \geq 1$ for $2 \leq k \leq N$. For any time-discrete function $w(t_k)$, the backward difference operator is given by $\triangledown_\tau w(t_k) := w(t_k) - w(t_{k-1})$ for $1 \leq k \leq N$.

For $n \geq 2$ and $1 \leq k \leq n-1$, the standard quadratic Lagrangian interpolation polynomial is employed:
\begin{align*}
L_{2,k}(t) = \sum_{j=-1}^{1} w(t_{k+j}) \prod_{\substack{l=-1 \ l \neq j}}^{1} \frac{t - t_{k+l}}{t_{k+j} - t_{k+l}}, \quad k = 1, 2, \ldots, n-1,
\end{align*}
which interpolates at the nodes $(t_{k-1}, w(t_{k-1}))$, $(t_k, w(t_k))$, and $(t_{k+1}, w(t_{k+1}))$. For $k = n$, the polynomial $L_{2,n-1}(t)$ is used instead. The derivative of the interpolation polynomial takes the form
\begin{align*}
    L'_{2,k}(t) &= \frac{\nabla_\tau w(t_k)}{\tau_k} + \frac{2t - t_{k-1} - t_k}{\tau_k + \tau_{k+1}} \left[\frac{\nabla_\tau w(t_{k+1})}{\tau_{k+1}} - \frac{\nabla_\tau w(t_k)}{\tau_k} \right] \\
    &= \frac{\nabla_\tau w(t_{k+1})}{\tau_{k+1}} + \frac{2t - t_{k+1} - t_k}{\tau_k + \tau_{k+1}} \left[ \frac{\nabla_\tau w(t_{k+1})}{\tau_{k+1}} - \frac{\nabla_\tau w(t_k)}{\tau_k} \right].
\end{align*}

Hence, the Caputo derivative $\partial^\alpha_t w(t)$ evaluated at $t_n$ for $2 \leq n \leq N$ is approximated through the definition (\ref{eq.1.2}) as follows
\begin{equation}\label{eq.2.1}
    \begin{aligned}
        \partial^\alpha_t w(t_n) &= \frac{1}{\Gamma(1-\alpha)} \left[ \sum_{k=1}^{n-1} \int_{t_{k-1}}^{t_k} (t_n - s)^{-\alpha} L'_{2,k}(s) \, \mathrm{d}s + \int_{t_{n-1}}^{t_n} (t_n - s)^{-\alpha} L'_{2,n-1}(s) \, \mathrm{d}s \right] + R_t^n \\
        &= \sum_{k=1}^{n} c_{n-k}^{(n)} \nabla_\tau w(t_k) + \frac{\rho_n}{1 + \rho_n} d_0^{(n)} \bigg[\nabla_\tau w(t_n) - \rho_n \nabla_\tau w(t_{n-1}) \bigg] \\
        &\quad + \sum_{k=1}^{n-1} d_{n-k}^{(n)} \frac{\nabla_\tau w(t_{k+1}) - \rho_{k+1} \nabla_\tau w(t_k)}{\rho_{k+1} (1 + \rho_{k+1}) } + R_t^n,
    \end{aligned}
\end{equation}
where the coefficients are given by
\begin{equation}\label{eq.2.2}
	\begin{aligned}
		c_{n-k}^{(n)} &= \frac{(t_n - t_{k-1})^{1-\alpha} - (t_n - t_k)^{1-\alpha}}{\tau_k \Gamma(2-\alpha)}, \\
		d_{n-k}^{(n)} &= \frac{2 \left[ (t_n - t_{k-1})^{2-\alpha} - (t_n - t_k)^{2-\alpha} \right]}{\tau_k^2 \Gamma(3-\alpha)} - \frac{(t_n - t_{k-1})^{1-\alpha} + (t_n - t_k)^{1-\alpha}}{\tau_k \Gamma(2-\alpha)}.
	\end{aligned}
\end{equation}

For the special case when $n=1$, the linear Lagrangian interpolation polynomial is used:
\begin{align*}
L_{1,1}(t) &= \frac{t - t_1}{t_0 - t_1}w(t_0) + \frac{t - t_0}{t_1 - t_0}w(t_1),
\end{align*}
yielding the approximation
\begin{equation}\label{eq.2.3}
	\begin{aligned}
		\partial^\alpha_t w(t_1) &= \frac{1}{\Gamma(1-\alpha)} \int_{t_0}^{t_1} (t_1 - s)^{-\alpha} L'_{1,1}(s)  \mathrm{d}s + R_t^1
		= c_0^{(1)} \triangledown_\tau w(t_1) + R_t^1,
	\end{aligned}
\end{equation}
where
\begin{align*}
c_0^{(1)} = \frac{1}{\tau_1^\alpha \Gamma(2-\alpha)}.
\end{align*}

By reorganizing the terms in (\ref{eq.2.1}) and (\ref{eq.2.3}), we obtain the unified form
\begin{equation}\label{eq.2.5}
\partial^\alpha_t w(t_n) = \sum_{k=1}^n B_{n-k}^{(n)}
 \triangledown_\tau w(t_k) + R_t^n
 =:\sum_{k=1}^n B_{n-k}^{(n)}
 \triangledown_\tau w^k, \\
\end{equation}
where \( w^k \) denotes the numerical approximation of
 \( w(t_k) \), and the discrete kernels \( B_{n-k}^{(n)} \)
 are defined as follows:
\begin{equation}\label{eq.2.6}
	B_{n-k}^{(n)} =
	\begin{cases}
		c_0^{(1)}, & k = 1, \, n = 1, \\
		c_1^{(2)} - \dfrac{d_1^{(2)}}{1 + \rho_2} - \dfrac{\rho_2^2}{1 + \rho_2} d_0^{(2)}, & k = 1, \, n = 2, \\
		c_0^{(n)} + \dfrac{d_1^{(n)}}{\rho_n (1 + \rho_n)} + \dfrac{\rho_n}{1 + \rho_n} d_0^{(n)}, & k = n, \, n \geq 2, \\
		c_{n-1}^{(n)} - \dfrac{d_{n-1}^{(n)}}{1 + \rho_2}, & k = 1, \, n \geq 3, \\
		c_1^{(n)} + \dfrac{d_2^{(n)}}{\rho_{n-1} (1 + \rho_{n-1})} - \dfrac{d_1^{(n)}}{1 + \rho_n} - \dfrac{\rho_n^2}{1 + \rho_n} d_0^{(n)}, & k = n-1, \, n \geq 3, \\
		c_{n-k}^{(n)} + \dfrac{d_{n-k+1}^{(n)}}{\rho_k (1 + \rho_k)} - \dfrac{d_{n-k}^{(n)}}{1 + \rho_{k+1}}, & 2 \leq k \leq n-2, \, n \geq 4.
	\end{cases}
\end{equation}

Next, we carefully examine the truncation error of the above approximation formula, with the precise result given in the following theorem:
\begin{theorem}\label{Th.2.1}
For any function $w(t) \in C^3[0, T ]$, the truncation error of the Caputo derivative approximation satisfies
\begin{align*}
\left|R_{t}^n\right| &= \left| \partial^\alpha_t w(t_{n}) - \sum_{k=1}^{n} B_{n-k}^{(n)}
\triangledown_\tau w(t_k) \right|
\leq \begin{cases}
\dfrac{\alpha}{2\Gamma(3-\alpha)} \max\limits_{t_0 \leq t \leq t_1} |w''(t)| \tau_{1}^{2-\alpha}, & n=1, \\[10pt]
\dfrac{3\alpha+1}{12\Gamma(2-\alpha)} \max\limits_{t_0 \leq t \leq t_n} |w'''(t)| \tau_{\text{max}}^{3-\alpha}, & n \geq 2.
\end{cases}
\end{align*}
\end{theorem}

\begin{proof}
The proof proceeds by considering two cases based on the value of $n$. For the case when $n=1$, starting from equation (\ref{eq.2.3}), we obtain
\begin{align*}
		\left|R_t^1\right|&= \left| \frac{1}{\Gamma(1-\alpha)} \int_{t_0}^{t_1} (t_1 - s)^{-\alpha} w'(s)  \mathrm{d}s - \frac{1}{\Gamma(1-\alpha)} \int_{t_0}^{t_1} (t_1 - s)^{-\alpha} L'_{1,1}(s)  \mathrm{d}s \right| \\
		&= \frac{1}{\Gamma(1-\alpha)} \left| \int_{t_0}^{t_1} (t_1 - s)^{-\alpha} \left[ w(s) - L_{1,1}(s) \right]'  \mathrm{d}s \right|.
	\end{align*}
Following the result from \cite{Gao2014}, the interpolation error can be expressed as
\begin{align*}
w(t)-L_{1,1}(t)=\frac{1}{2}w''(\eta_0)(t-t_0)(t-t_1),\;
 \eta_0\in (t_0,t_1),\;t\in[t_0,t_1].
\end{align*}
Applying integration by parts yields
	\begin{align*}
		\left|R_t^1\right| &= \frac{1}{\Gamma(1-\alpha)} \left| -\alpha \int_{t_0}^{t_1} (t_1 - s)^{-\alpha-1} \left[ w(s) - L_{1,1}(s) \right]  \mathrm{d}s \right| \\
		&= \frac{\alpha}{2\Gamma(1-\alpha)} \left| \int_{t_0}^{t_1} w''(\eta_0) (s - t_0) (t_1 - s)^{-\alpha}  \mathrm{d}s \right|.
	\end{align*}
	Since the integrand maintains constant sign over the interval, we deduce
	\begin{align*}
		\left|R_t^1\right|&\leq \frac{\alpha}{2\Gamma(1-\alpha)} \max_{t_0 \leq t \leq t_1} |w''(t)| \int_{t_0}^{t_1} (s - t_0) (t_1 - s)^{-\alpha}  \mathrm{d}s.
	\end{align*}
The evaluation of the integral
$
\int_{t_0}^{t_1} (s - t_0)(t_1 - s)^{-\alpha} \, \mathrm{d}s = \frac{\Gamma(1-\alpha)}{\Gamma(3-\alpha)} \tau_1^{2-\alpha},
$
directly implies that the truncation error at the first time step satisfies the bound
\[
\left|R_t^1\right| \leq \frac{\alpha}{2\Gamma(3-\alpha)} \max_{t_0 \leq t \leq t_1} |w''(t)| \, \tau_1^{2-\alpha},
\]
	
For the case when $n \ge 2$, we begin with equation (\ref{eq.2.1}) to derive
\begin{align*}
    \left|R_t^n\right| =&\left| \frac{1}{\Gamma(1-\alpha)}\int_{t_{0}}^{t_{n}} \left(t_{n}-s\right)^{-\alpha}
    w'(s)\mathrm{d}s \right. \\ &\left.-\frac{1}{\Gamma(1-\alpha)}\left[\sum\limits_{k=1}^{n-1}
     \int_{t_{k-1}}^{t_k} (t_{n}-s)^{-\alpha} L'_{2,k}(s) \mathrm{d}s+\int_{t_{n-1}}^{t_{n}} (t_{n}-s)^{-\alpha}
     L'_{2,n-1}(s) \mathrm{d}s \right]  \right|\\
    =& \frac{1}{\Gamma(1-\alpha)} \Bigg{| }  \int_{t_{n-1}}^{t_n} \left(t_{n}-s\right)^{-\alpha}\left[w(s)-L_{2,n-1}(s) \right]'\mathrm{d}s  \\
    &+\sum\limits_{k=1}^{n-1} \int_{t_{k-1}}^{t_k}\left(t_{n}-s\right)^{-\alpha}\left[w(s)-L_{2,k}(s) \right]'\mathrm{d}s \Bigg{| } .
\end{align*}
Using the interpolation error estimate
\begin{align*}
 w(t)-L_{2,k}(t)=\frac{1}{6}w'''(\eta_k)(t-t_{k-1})(t-t_k)(t-t_{k+1}),\; \eta_{k} \in (t_{k-1},t_{k+1}),\;t\in[t_{k-1},t_{k+1}],
\end{align*}
we obtain the following estimate:

	\begin{align*}
		\left|R_t^n\right|&= \frac{1}{\Gamma(1-\alpha)} \Bigg| -\alpha \int_{t_{n-1}}^{t_n} (t_n - s)^{-\alpha-1} \left[ w(s) - L_{2,n-1}(s) \right]  \mathrm{d}s \\
		&\quad - \alpha \sum_{k=1}^{n-1} \int_{t_{k-1}}^{t_k} (t_n - s)^{-\alpha-1} \left[ w(s) - L_{2,k}(s) \right]  \mathrm{d}s \Bigg| \\
		&\leq \frac{\alpha}{6\Gamma(1-\alpha)} \max_{t_0 \leq t \leq t_n} |w'''(t)| \Bigg\{ \int_{t_{n-1}}^{t_n} \left|(s - t_{n-2})(s - t_{n-1})\right| (t_n - s)^{-\alpha}  \mathrm{d}s \\
		&\quad + \sum_{k=1}^{n-1} \int_{t_{k-1}}^{t_k} \left|(s - t_{k-1})(t_k - s)(t_{k+1} - s)\right| (t_n - s)^{-\alpha-1}  \mathrm{d}s \Bigg\}.
	\end{align*}
	Bounding the polynomial factors in the integrands, we note that
	$
		|(s - t_{n-2})(s - t_{n-1})| \leq (\tau_{n-1} + \tau_n) \tau_n,
	$
	 for $s \in [t_{n-1}, t_n]$ and
	$
		|(s - t_{k-1})(t_k - s)(t_{k+1} - s)| \leq \frac{\tau_k^2}{4} (\tau_k + \tau_{k+1})
$ for $s \in [t_{k-1}, t_k]$. Substituting these estimates yields
	\begin{align*}
		\left|R_t^n\right| &\leq \frac{\alpha}{6\Gamma(1-\alpha)} \max_{t_0 \leq t \leq t_n} \left|w'''(t)\right| \Bigg\{ (\tau_{n-1} + \tau_n) \tau_n \int_{t_{n-1}}^{t_n} (t_n - s)^{-\alpha}  \mathrm{d}s \vspace{0.3cm}\\
		&\quad + \sum_{k=1}^{n-1} \frac{\tau_k^2}{4} (\tau_k + \tau_{k+1}) \int_{t_{k-1}}^{t_k} (t_n - s)^{-\alpha-1}  \mathrm{d}s \Bigg\}\vspace{0.3cm} \\
		&= \frac{\alpha}{6\Gamma(1-\alpha)} \max_{t_0 \leq t \leq t_n}
 \left|w'''(t)\right| \Bigg\{ \frac{(\tau_{n-1} + \tau_n) \tau_n^{2-\alpha}}{1 - \alpha}\vspace{0.3cm} \\
		&\quad + \sum_{k=1}^{n-1}
\frac{(\tau_k + \tau_{k+1}) \tau_k^2\left[(t_n - t_k)^{-\alpha} - (t_n - t_{k-1})^{-\alpha}\right]}{4\alpha} \Bigg\} \vspace{0.3cm}\\
		&\leq \frac{3\alpha + 1}{12\Gamma(2-\alpha)} \max_{t_0 \leq t \leq t_n}
\left|w'''(t)\right| \tau_{\max}^{3-\alpha}.
	\end{align*}
This completes the proof of the theorem.
\end{proof}

To facilitate subsequent analysis, the expression (\ref{eq.2.5}) requires reformulation and reorganization. Building upon the decomposition framework presented in \cite{Liao2024}, an improved scheme is developed as follows:
\begin{equation}\label{eq.2.7}
\begin{aligned}\displaystyle
\partial^\alpha_t w(t_{n}) =\left( \theta c_{0}^{(n)}+ \frac{\rho_{n}}{1+\rho_{n}}d_{0}^{(n)} \right) \triangledown_\tau w^{n}
-\frac{\rho_{n}^2}{1+\rho_{n}}d_{0}^{(n)} \triangledown_\tau w^{n-1} +\sum\limits_{k=1}^{n} \widetilde{c}_{n-k}^{(n)}\triangledown_\tau w^{k}.
\end{aligned}
\end{equation}
The crucial aspect of this formulation lies in the parameter selection,
particularly with $$\theta=\frac{1}{2-\alpha}+\frac{2^{1-\alpha}
\alpha^2+\alpha-2\alpha^2}{2(2-\alpha)(1+\overline{\rho})},\;\overline{\rho} \approx 4.7476114,$$ which differs significantly from the simpler coefficient $\frac{1}{2-\alpha}$ employed in \cite{Liao2024}. The discrete kernels $\widetilde{c}_{n-k}^{(n)}$ are carefully constructed through
the following piecewise definition:
\begin{equation}\label{eq.2.8}
\begin{aligned}\displaystyle
\widetilde{c}_{n-k}^{(n)}=\begin{cases}
(1-\theta)c_0^{(1)}, & \mathrm{if} \quad k=1,\;n=1, \\
c_{n-1}^{(n)}-\frac{1}{1+\rho_{2}}d_{n-1}^{(n)} , & \mathrm{if} \quad k=1,\;n\ge 2,\\
(1-\theta)c_{0}^{(n)} +\frac{1} {\rho_{n}\left(1+\rho_{n} \right)} d_{1}^{(n)} , & \mathrm{if} \quad k=n,\;n\ge 2,\\
c_{n-k}^{(n)}+\frac{1}{\rho_{k}\left(1+\rho_{k}\right) }d_{n-k+1}^{(n)} -\frac{1}{1+\rho_{k+1}}d_{n-k}^{(n)} , & \mathrm{if} \quad 2\leq k \leq n-1,\; n\ge 3.
\end{cases}
\end{aligned}
\end{equation}
This modified decomposition offers several analytical advantages over previous approaches. The parameter $\theta$ incorporates higher-order correction terms that account for both the fractional order $\alpha$ and the characteristic value $\overline{\rho}$. The piecewise definition of the discrete kernels ensures proper handling of temporal discretization across different time steps, with special attention given to boundary cases and intermediate temporal nodes. The formulation maintains consistency while improving numerical stability through its balanced treatment of current and previous time step contributions.

To support the theoretical analysis of the finite difference scheme to be developed, we conclude this section by introducing the following key lemmas.
\begin{lemma}\label{Le.2.2}
 For any $\alpha \in (0,1)$, the following inequality holds
\begin{align*}
q(z,y,\alpha):=2\theta(2-\alpha)+\frac{2\alpha z-\alpha z^{2-\frac{\alpha}{2}}}{1+z}-\frac{\alpha y^{2-\frac{\alpha}{2}}}{1+y}>0 \quad \text{for}\; 1\leq z,y<\rho^*,
\end{align*}
where $\rho^* = \rho^*(\alpha)$ represents the unique root of the equation $$\frac{1+\rho^*}{\alpha}+\rho^*-(\rho^*)^{2-\frac{\alpha}{2}}
+2^{-\alpha}\alpha+\frac{1}{2}-\alpha=0.$$
\end{lemma}

\begin{proof}
The proof proceeds by analyzing the behavior of the function $q(z,y,\alpha)$. First, differentiating with respect to $z$ yields
\begin{align*}
\partial_z q(z,y,\alpha)=\frac{\alpha}{2(1+z)^2}\left[4-(4-\alpha)z^{1-\frac{\alpha}{2}}
-(2-\alpha)z^{2-\frac{\alpha}{2}} \right].
\end{align*}
Consider the auxiliary function $q_1(z,\alpha)=4-(4-\alpha)z^{1-\frac{\alpha}{2}}-(2-\alpha)z^{2-\frac{\alpha}{2}}$. Analysis shows that $q_1(z,\alpha)$ decreases monotonically for $z>0$, with $q_1(1,\alpha)=2\alpha-2<0$. Consequently, $q(z,y,\alpha)$ decreases in $z$ for $z\in(1,+\infty)$. Furthermore, the derivative with respect to $y$
\begin{align*}
\partial_y q(z,y,\alpha)=\frac{\alpha y^{1-\frac{\alpha}{2}}}{2(1+y)^2}\left[\alpha-4-(2-\alpha)y \right]<0 \quad \text{for}\;y>0
\end{align*}
demonstrates that $q(z,y,\alpha)$ also decreases in $y$ for $y\in(1,+\infty)$. Therefore, we obtain
\begin{align*}
        q(z,y,\alpha) > q(\rho^*,\rho^*,\alpha) = 2\theta(2-\alpha) + \frac{2\alpha \rho^* - 2\alpha (\rho^*)^{2-\frac{\alpha}{2}}}{1+\rho^*}.
    \end{align*}

To complete the proof, we examine the auxiliary function
\begin{align*}
    q_2(\rho,\alpha)=\frac{1+\rho}{\alpha}+\rho-\rho^{2-\frac{\alpha}{2}}
    +2^{-\alpha}\alpha+\frac{1}{2}-\alpha \quad \text{for}\; \;\rho \ge 1,
\end{align*}
whose derivative is
  \begin{align*}
        \partial_{\rho} q_2(\rho,\alpha) = 1 + \frac{1}{\alpha} - \left(2 - \frac{\alpha}{2}\right) \rho^{1-\frac{\alpha}{2}}.
    \end{align*}
The equation $\partial_{\rho} q_2(\rho,\alpha)=0$ possesses a unique root $g>1$ since $\partial_{\rho} q_2(1,\alpha)>0$. Thus, for any $\alpha \in (0,1)$, $q_2(\rho,\alpha)$ increases on $(1,g)$ and decreases on $(g,+\infty)$. Notably,
\begin{align*}
        q_2(g,\alpha) = \frac{1}{\alpha} + g\left(1 + \frac{1}{\alpha}\right)\left(\frac{2-\alpha}{4-\alpha}\right) + 2^{-\alpha}\alpha + \frac{1}{2} - \alpha > 0.
    \end{align*}
while $q_2(+\infty,\alpha)<0$. This guarantees the existence of a unique $\rho^* \in (g,+\infty)$ satisfying $q_2(\rho^*,\alpha)=0$ for any $\alpha \in (0,1)$.
Combining these results yields $q(z,y,\alpha)>q(\rho^*,\rho^*,\alpha)
=\frac{2\alpha}{1+\rho^*}q_2(\rho^*,\alpha)=0$, which completes the proof.
\end{proof}

\begin{remark}\label{Re.1}
Numerical verification shows that $\rho^*(1) \approx 4.864$ and $\rho^*(\alpha) \to +\infty$ as $\alpha \to 0^+$. However, obtaining an explicit solution $\rho^*(\alpha)$ satisfying $q_2(\rho^*,\alpha)=0$ remains challenging. By differentiating the equation with respect to $\alpha$, we derive
\begin{align*}
        \frac{\mathrm{d} \rho^*(\alpha)}{\mathrm{d} \alpha} = & \frac{\dfrac{(\rho^*)^{2-\frac{\alpha}{2}}}{2} \ln\rho^* - \dfrac{1+\rho^*}{\alpha^2} + 2^{-\alpha} - 2^{-\alpha}\alpha \ln 2 - 1}{\left(2-\frac{\alpha}{2}\right)\left(\rho^*\right)^{1-\frac{\alpha}{2}} - \frac{1}{\alpha} - 1} \\
        &= \frac{\alpha \ln\rho^* - 2 \dfrac{1 + \rho^* - 2^{-\alpha}\alpha^2 + 2^{-\alpha}\alpha^3 \ln 2 + \alpha^2}{1 + \rho^* + \alpha \rho^* + 2^{-\alpha}\alpha^2 + \frac{\alpha}{2} - \alpha^2}}{(4-\alpha)\alpha (\rho^*)^{1-\frac{\alpha}{2}} - 2 - 2\alpha} (\rho^*)^{2-\frac{\alpha}{2}}.
    \end{align*}
To analyze this expression, consider the auxiliary function
\[ q_3(\rho^*,\alpha)=\alpha \mathrm{ln}\rho^*-2\frac{1+\rho^*-2^{-\alpha}\alpha^2
	+2^{-\alpha}\alpha^3 \mathrm{ln}2+\alpha^2}{1+\rho^*+\alpha \rho^*+2^{-\alpha}\alpha^2+\frac{\alpha}{2}-\alpha^2}. \]
The complexity of $q_3(\rho^*,\alpha)$ makes it difficult to determine its monotonicity through direct differentiation. Instead, numerical investigation using Matlab reveals how this function varies with increasing $\rho^*$ for different values of $\alpha$. Figure \ref{fig:2.1} demonstrates that $q_3(\rho^*,\alpha)$ increases monotonically with respect to $\rho^*$. Furthermore, analysis confirms that $q_3(\rho^*,\alpha)=0$ possesses a unique positive root for any $\alpha \in (0,1)$.
Numerical solution of the equations $q_2(\rho^*,\alpha)=0$ and $q_3(\rho^*,\alpha)=0$ yields the approximate solution pair $(\rho^*,\alpha) \approx (4.7476114, 0.82265)$. Denoting $\overline{\rho} \approx 4.7476114$, examination shows that $q_2(\overline{\rho},\alpha)$ attains its minimum value at this point, with $q_2(\overline{\rho},\alpha) \ge q_2(\overline{\rho},0.82265) \approx 9.3942\times10^{-8}$. The functional relationship between $\rho^*$ and $\alpha$ in the equation $q_2(\rho^*,\alpha)=0$ is illustrated in Figure \ref{fig:2.2}.

\begin{figure}[htbp]
\centering
\includegraphics[width=0.6\linewidth]{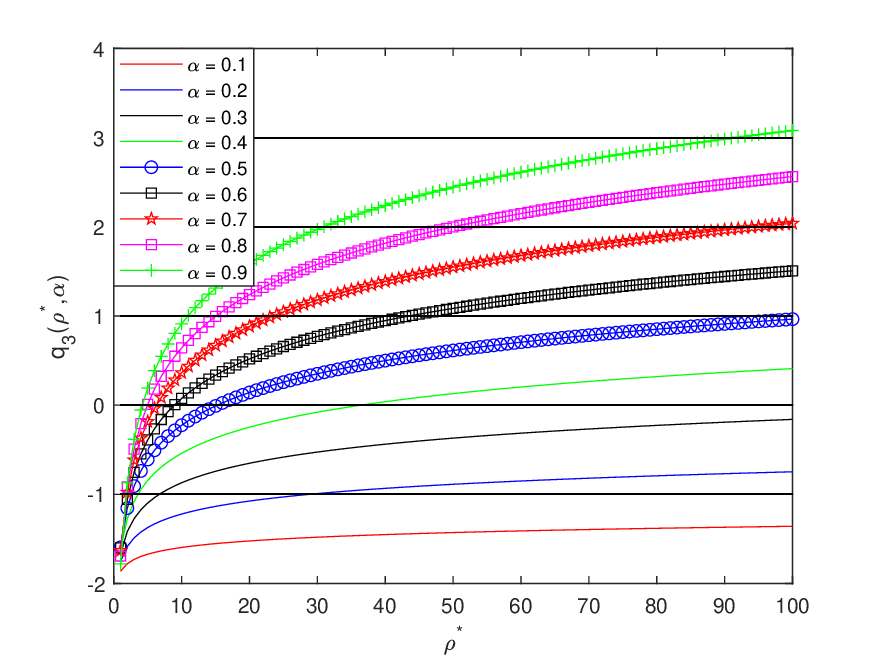}
\caption{Variation of $q_3(\rho^*,\alpha)$ with respect to $\rho^*$ for different values of $\alpha$.}
\label{fig:2.1}
\end{figure}

\begin{figure}[htbp]
\centering
\includegraphics[width=0.6\linewidth]{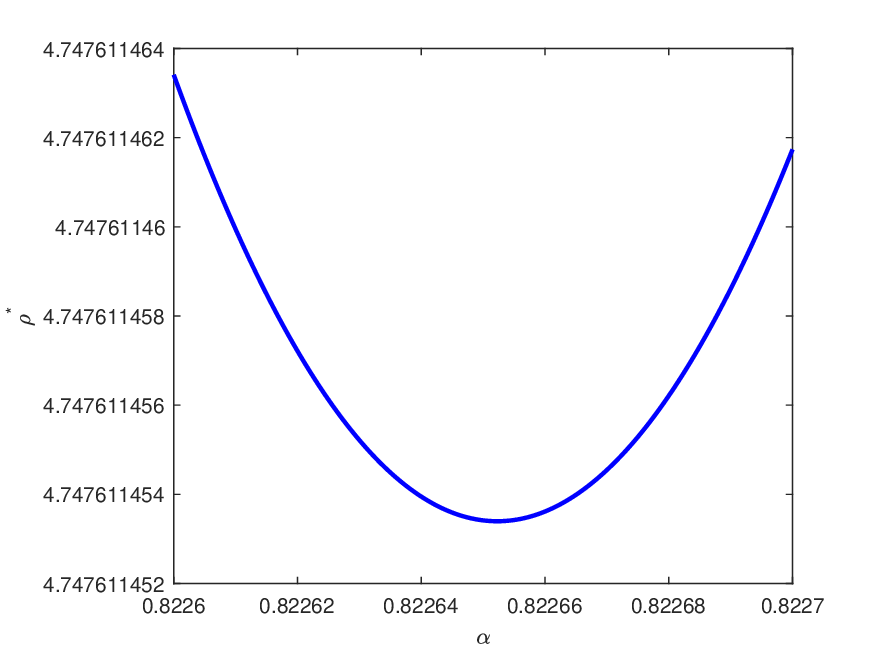}
\caption{Functional dependence between $\rho^*$ and $\alpha$ satisfying $q_2(\rho^*,\alpha)=0$.}
\label{fig:2.2}
\end{figure}
\end{remark}

\begin{lemma}\label{Le.2.3}
  Let $\rho_{n} \leq \rho^*(\alpha)$ hold for all $2\leq n \leq N$. Then the following inequality is valid:
   	\begin{align*}
&	\Bigg{(}	\Big{(} \theta c_{0}^{(n)}+  \frac{\rho_{n}}{1+\rho_{n}}d_{0}^{(n)} \Big{)} \triangledown_\tau w^{n}
		-\frac{\rho_{n}^2}{1+\rho_{n}}d_{0}^{(n)} \triangledown_\tau w^{n-1} \Bigg{)} \triangledown_\tau w^{n}  \\
\ge &\frac{\alpha \rho_{n+1}^{2-\frac{\alpha}{2}} \left( \triangledown_\tau w^{n}\right)^2} {2 \left(1+\rho_{n+1}\right)\tau_{n}^\alpha \Gamma(3-\alpha)} -\frac{\alpha \rho_{n}^{2-\frac{\alpha}{2}} \left( \triangledown_\tau w^{n-1} \right)^2} {2 \left(1+\rho_{n}\right)\tau_{n-1}^\alpha \Gamma(3-\alpha)}+\frac{q(\rho_n,\rho_{n+1},\alpha) \left(\triangledown_\tau w^{n} \right)^2}{2\tau_{n}^\alpha \Gamma(3-\alpha)},
	\end{align*}
 where the quantity $\rho^*(\alpha)$ represents the upper bound defined in Lemma \ref{Le.2.2}, and $q(z,y,\alpha)$ denotes the function introduced in the same lemma.
\end{lemma}

\begin{proof}
The proof begins by expanding the left-hand side using the definition (\ref{eq.2.2}), which yields
\begin{align*}
& \Bigg{(} \Big{(} \theta c_{0}^{(n)}+ \frac{\rho_{n}}{1+\rho_{n}}d_{0}^{(n)} \Big{)} \triangledown_\tau w^{n}
-\frac{\rho_{n}^2}{1+\rho_{n}}d_{0}^{(n)} \triangledown_\tau w^{n-1} \Bigg{)} \triangledown_\tau w^{n} \\
= & \frac{\theta \left(\triangledown_\tau w^{n} \right)^2 }{\tau_{n}^\alpha \Gamma(2-\alpha)} +\frac{\alpha \rho_{n} \left( \triangledown_\tau w^{n} \right)^2}{ \left(1+\rho_{n}\right)\tau_{n}^\alpha \Gamma(3-\alpha)} -\frac{\alpha \rho_{n}^2 \triangledown_\tau w^{n}\cdot \triangledown_\tau w^{n-1} }{ \left(1+\rho_{n}\right)\tau_{n}^\alpha \Gamma(3-\alpha)}.
\end{align*}
An application of Young's inequality provides the estimate
\begin{align*}
	\triangledown_\tau w^{n}\cdot \triangledown_\tau w^{n-1}
	&=\frac{\tau_{n-1}^{\frac{\alpha}{4}}}{\tau_{n}^{\frac{\alpha}{4}}}\triangledown_\tau w^{n}\cdot \frac{\tau_{n}^{\frac{\alpha}{4}}}
	{\tau_{n-1}^{\frac{\alpha}{4}}}\triangledown_\tau w^{n-1}
   \leq \frac{1}{2} \left( \frac{\tau_{n-1}^{\frac{\alpha}{4}}}
   {\tau_{n}^{\frac{\alpha}{4}}}\triangledown_\tau w^{n}\right)^2 +\frac{1}{2} \left( \frac{\tau_{n}^{\frac{\alpha}{4}}}
   {\tau_{n-1}^{\frac{\alpha}{4}}}\triangledown_\tau w^{n-1} \right)^2 \\
   =&\frac{\tau_{n-1}^{\frac{\alpha}{2}}}
   {2\tau_{n}^{\frac{\alpha}{2}}}
   \left( \triangledown_\tau w^{n} \right)^2 +\frac{\tau_{n}^{\frac{\alpha}{2}}}{2\tau_{n-1}^{\frac{\alpha}{2}}} \left( \triangledown_\tau w^{n-1} \right)^2.
\end{align*}
Substituting this inequality into the previous expression leads to the lower bound
\begin{align*}
& \Bigg{(} \Big{(} \theta c_{0}^{(n)}+ \frac{\rho_{n}}{1+\rho_{n}}d_{0}^{(n)} \Big{)} \triangledown_\tau w^{n}
-\frac{\rho_{n}^2}{1+\rho_{n}}d_{0}^{(n)} \triangledown_\tau w^{n-1} \Bigg{)} \triangledown_\tau w^{n} \\
\ge & \frac{\theta \left(\triangledown_\tau w^{n} \right)^2 }{\tau_{n}^\alpha \Gamma(2-\alpha)}+\frac{\alpha \rho_{n} \left(\triangledown_\tau w^{n} \right)^2}
{ \left(1+\rho_{n}\right)\tau_{n}^\alpha \Gamma(3-\alpha)} -\frac{\alpha \rho_{n}^{2-\frac{\alpha}{2}} \left(\triangledown_\tau w^{n} \right)^2 }{ 2 \left(1+\rho_{n}\right)\tau_{n}^\alpha \Gamma(3-\alpha)} -\frac{\alpha \rho_{n}^{2-\frac{\alpha}{2}} \left(\triangledown_\tau w^{n-1} \right)^2 }{ 2 \left(1+\rho_{n}\right)\tau_{n-1}^\alpha \Gamma(3-\alpha)} \\
= &\frac{\alpha \rho_{n+1}^{2-\frac{\alpha}{2}} \left(\triangledown_\tau w^{n} \right)^2}{2 \left(1+\rho_{n+1}\right)\tau_{n}^\alpha \Gamma(3-\alpha)} -\frac{\alpha \rho_{n}^{2-\frac{\alpha}{2}} \left(\triangledown_\tau w^{n-1} \right)^2}{2 \left(1+\rho_{n}\right)\tau_{n-1}^\alpha \Gamma(3-\alpha)}+\frac{q(\rho_n,\rho_{n+1},\alpha) \left(\triangledown_\tau w^{n} \right)^2}{2\tau_{n}^\alpha \Gamma(3-\alpha)}.
\end{align*}
This completes the proof of the lemma.
\end{proof}

Building upon these results, we establish the following lemma which extends previous work in \cite{Ji}, \cite{Quan2}, and \cite{Liao2024}:

\begin{lemma}\label{Le.2.4}
For an arbitrary fixed index $n \ge 2$ and given positive kernels $\left\{\chi_{n-j}^{(n)}\right\}_{j=1}^{n}$, consider the auxiliary kernels defined by $\mathbf{a}_{0}^{(n)} := (2 - \sigma_{min}) \chi_{0}^{(n)}$ and
$\mathbf{a}_{n-j}^{(n)} := \chi_{n-j}^{(n)}$ for $1\leq j\leq n-1$. Under the assumption that there exists a constant $\sigma_{min}\in [0, 2)$ such that the modified kernels $\mathbf{a}_{n-j}^{(n)}$ satisfy the following conditions
\begin{itemize}
\item[(1)] Monotonicity along rows:  $\mathbf{a}_{n-j-1}^{(n)} \ge \mathbf{a}_{n-j}^{(n)} > 0$ holds for all  $1\leq j\leq n-1$;
\item[(2)] Monotonicity along columns:  $\mathbf{a}_{n-j-1}^{(n-1)} \ge \mathbf{a}_{n-j}^{(n)}$ whenever $1\leq j\leq n-1$;
\item[(3)] Convexity property: The inequality $\mathbf{a}_{n-j-2}^{(n-1)}-\mathbf{a}_{n-j-1}^{(n-1)} \ge \mathbf{a}_{n-j-1}^{(n)} - \mathbf{a}_{n-j}^{(n)}$ is satisfied for  $1\leq j\leq n-2$;
\end{itemize}
then for any real sequence $\left\{\phi_{k}\right\}_{k=1}^{n}$, the discrete gradient structure takes the form
\begin{align*}
2\phi_n\sum_{j=1}^n\chi_{n-j}^{(n)}\phi_j = Y[\vec{\phi}_n] - Y[\vec{\phi}_{n-1}] + \sigma_{\min}\chi_0^{(n)}\phi_n^2 + Y_R[\vec{\phi}_n]\quad\text{for}\quad n\ge 1,
\end{align*}
where $\vec{\phi}_n = (\phi_n,\phi_{n-1},\cdots,\phi_1)^T$. The non-negative quadratic functionals $Y$ and $Y_R$ are defined respectively as
\begin{align*}
Y[\vec{\phi}_n] &:= \sum_{j=1}^{n-1}\left(\mathbf{a}_{n-j-1}^{(n)} - \mathbf{a}_{n-j}^{(n)}\right)\bigg(\sum_{\ell=j+1}^n \phi_\ell\bigg)^2 + \mathbf{a}_{n-1}^{(n)}\bigg(\sum_{\ell=1}^n \phi_\ell\bigg)^2\quad\text{for } n\ge 1,
\end{align*}
and
\begin{align*}
Y_R[\vec{\phi}_n] &:= \sum_{j=1}^{n-2}\left(\mathbf{a}_{n-j-2}^{(n-1)} - \mathbf{a}_{n-j-1}^{(n-1)} - \mathbf{a}_{n-j-1}^{(n)} + \mathbf{a}_{n-j}^{(n)}\right)\left(\sum_{\ell=j+1}^{n-1}\phi_\ell\right)^2\\&
 +\left(\mathbf{a}_{n-2}^{(n-1)} - \mathbf{a}_{n-1}^{(n)}\right)
 \left(\sum_{\ell=1}^{n-1} \phi_\ell\right)^2\quad\text{for } n\ge 2.
\end{align*}
Moreover, the convolution kernels $\chi_{n-k}^{(n)}$ exhibit positive definiteness through the
\begin{align*}
  2 \sum_{k=1}^n \phi_k \sum_{j=1}^k \chi_{k-j}^{(k)} \phi_j \ge Y[\vec{\phi}_n] + \sigma_{\min} \sum_{k=1}^n \chi_0^{(k)} \phi_k^2 \quad \text{for } n\ge 1.
\end{align*}
\end{lemma}

We now present the discrete gradient structure by establishing key properties of the auxiliary convolution kernels.

\begin{theorem}\label{Th.2.5}
Let the time-step ratios $\rho_n$ satisfy $\rho_n \ge 1$ for $n \ge 2$. For the discrete kernels $\widetilde{c}_{n-k}^{(n)}$ defined in (\ref{eq.2.8}), consider the auxiliary kernels constructed as follows:
\begin{align*}
J_0^{(n)} := 2\widetilde{c}_0^{(n)}, \;\;
J_{n-k}^{(n)} := \widetilde{c}_{n-k}^{(n)}\quad \text{for} \quad 1 \leq k \leq n-1.
\end{align*}
These auxiliary convolution kernels $J_{n-k}^{(n)}$ possess the following properties:
\begin{itemize}
\item[(1)] Monotonicity: $J_{n-k-1}^{(n)} \ge J_{n-k}^{(n)}$ holds for all $k =1,2, \cdots, n-1$ when $n\ge 2$;
\item[(2)] Convexity: The inequality $J_{n-k-2}^{(n-1)}-J_{n-k-1}^{(n-1)} \ge J_{n-k-1}^{(n)}-J_{n-k}^{(n)}$ is satisfied for $k =1,2, \cdots, n-2$ when $n\ge 3$;
\item[(3)] Dominance: $J_{n-k-1}^{(n-1)} \ge J_{n-k}^{(n)}$ holds for $k =1,2, \cdots, n-1$ when $n\ge 2$.
\end{itemize}
\end{theorem}

\begin{proof}
The structure of the coefficients $J_{n-k}^{(n)}$ closely resembles that of $A_{n-k}^{(n)}$ in reference \cite{Liao2024}. Consequently, the proof strategy employed in \cite{Liao2024} can be directly adapted to establish the properties of our coefficients. Since the referenced work provides a comprehensive derivation of these results, the detailed proof is omitted here for brevity.
\end{proof}

Building upon Lemma \ref{Le.2.3} and Theorem \ref{Th.2.5}, we establish the following important result regarding the discrete gradient structure.

\begin{theorem}\label{Th.2.6}
For any positive integer $n \geq 1$, let $\mathcal{G}$ be a non-negative functional defined by
\begin{align*}
        \mathcal{G}\left[\triangledown_\tau w^n\right] = \frac{\alpha \rho_{n+1}^{2-\frac{\alpha}{2}} \left(\triangledown_\tau w^n \right)^2}{2 \left(1+\rho_{n+1}\right)\tau_{n}^\alpha \Gamma(3-\alpha)}
        + \frac{1}{2} \sum_{j=1}^{n-1} \left( J_{n-j-1}^{(n)} - J_{n-j}^{(n)} \right) \left( w^n - w^j \right)^2
        + \frac{1}{2} J_{n-1}^{(n)} \left( w^n - w^0 \right)^2
    \end{align*}
Suppose the time-step ratios $\rho_{n}$ satisfy the constraint
\begin{align*}
1 \leq \rho_n \leq \rho^*(\alpha) \quad \text{for all} \quad n \geq 2,
\end{align*}
where $\rho^*(\alpha)$ denotes the upper bound exceeding $\overline{\rho} \approx 4.7476114$, as established in Lemma \ref{Le.2.2} and Remark \ref{Re.1}. Then, for every $2 \leq n \leq N$, the following inequality holds
  \begin{align*}
        \partial^\alpha_t w(t_{n}) \cdot \triangledown_\tau w^n \ge \mathcal{G}\left[\triangledown_\tau w^n\right] - \mathcal{G}\left[\triangledown_\tau w^{n-1}\right] + \frac{q(\rho_n,\rho_{n+1},\alpha) \left(\triangledown_\tau w^{n} \right)^2}{2\tau_{n}^\alpha \Gamma(3-\alpha)}.
    \end{align*}
\end{theorem}

\section{Construction of numerical scheme}\label{sec: Establishment of the numerical scheme}
Let $M$ be a positive integer representing the number of spatial discretization points over the domain $\Omega$. The interval $(a,b)$ is partitioned using uniformly distributed grid points $a=x_0<x_1<\cdots<x_M=b$, where each $x_i = a+ih$ with constant mesh size $h=(b-a)/M$. We define two discrete function spaces on this grid:
$$\mathcal{U}_h=\{u\;|\;u=(u_0,u_1,\cdots,u_{M-1},u_M)\},\quad \mathring{\mathcal{U}}_h=\{u\;|\;u\in \mathcal{U}_h,u_0=u_M=0\}.$$

For any sufficiently smooth function $u(x,\cdot) \in C^6[x_{i-1},x_{i+1}]$, the following fourth-order approximation holds \cite{Liao2010}:
\begin{align}\label{eq.3.1}
\mathscr{A} \partial_x^2u(x_i,\cdot) = \delta_x^2u(x_i,\cdot) + \frac{h^4}{360}\int_{0}^{1}\left[u^{(6)}(x_i-\lambda h,\cdot)+u^{(6)}(x_i+\lambda h,\cdot)\right]\eta(\lambda)\mathrm{d}\lambda,
\end{align}
where the weight function $\eta(\lambda)=(1-\lambda)^3[5-3(1-\lambda)^2]$ ensures the approximation accuracy. The discrete operators $\mathscr{A}$ and $\delta_x^2$ are respectively defined by
\begin{align*}
\mathscr{A}u(x_i,t)=\begin{cases}
\frac{1}{12}\left[u(x_{i-1},\cdot)
+10u(x_{i},\cdot) +u(x_{i+1},\cdot)\right], &\quad 1 \leq i \leq M-1,\\
u(x_{i},\cdot), &\quad i=0,M,
\end{cases}
\end{align*}
and
\begin{align*}
\delta_x^2u(x_i,\cdot)=\frac{u(x_{i-1},\cdot)
-2u(x_i,\cdot)+u(x_{i+1},\cdot)}{h^2}.
\end{align*}

By applying the averaging operator $\mathscr{A}$ to both sides of equation (\ref{eq.1.1}) at each grid point $(x_i,t_{n})$ and employing the approximation formula (\ref{eq.3.1}), we obtain the following semi-discrete system:
\begin{equation}\label{eq.3.2}
\left\{
\begin{aligned}
& \mathscr{A}\partial^\alpha_t u(x_i,t_{n}) =
\kappa \delta_x^2 f(u(x_i,t_{n})) +
 \kappa \varepsilon^2 \delta_x^2 v(x_i,t_{n})
 + \mathcal{O}\left(h^4\right), \\
& \mathscr{A} v(x_i,t_{n}) = -\delta_x^2 u(x_i,t_{n})
+ \mathcal{O}\left(h^4\right),
\end{aligned}
\right.
\end{equation}
which holds for all interior points $1\leq i \leq M-1$ and time levels $1\leq n \leq N$. The truncation errors maintain fourth-order accuracy in space, consistent with the approximation properties of the scheme.

Substituting (\ref{eq.2.5}) into (\ref{eq.3.2}) yields the discrete equation
\begin{equation}\label{eq.3.3}
\begin{aligned}
\mathscr{A} \sum_{k=1}^{n}B_{n-k}^{(n)} \left(u(x_i,t_k)-u(x_i,t_{k-1}) \right) = \kappa \delta_x^2 f(u(x_i,t_{n})) + \kappa \varepsilon^2 \delta_x^2 v(x_i,t_{n}) + r_{i}^{n},
\end{aligned}
\end{equation}
which holds for $1\leq i \leq M-1$ and $1\leq n \leq N$. The truncation error $r_i^n$ satisfies
\begin{equation}\label{eq.3.4}
\left|r_{i}^{n}\right| \leq C \left\{
\begin{aligned}
& \left(\tau^{2-\alpha}+h^4\right), \;\; 1\leq i \leq M-1, n=1, \\
& \left(\tau^{3-\alpha}+h^4\right), \;\; 1\leq i \leq M-1, 2\leq n \leq N,
\end{aligned}
\right.
\end{equation}
where $C$ denotes a positive constant
independent of the discretization parameters.
The initial condition from (\ref{eq.1.1})
provides the starting values
\begin{align*}
u(x_i,t_{0}) = u_0(x_i), \quad 1\leq i \leq M-1.
\end{align*}

By omitting the higher-order error term $r_{i}^{n}$ in (\ref{eq.3.3}) and replacing the exact solution $u(x_i,t_n)$ with its numerical approximation $u_{i}^n$, we derive the following numerical scheme for problem (\ref{eq.1.1}):
\begin{equation}\label{eq.3.5}
\begin{cases}\displaystyle
\mathscr{A} \sum\limits_{k=1}^{n}B_{n-k}^{(n)}
\triangledown_\tau u_i^k = \kappa \delta_x^2 f(u_i^n)
+ \kappa \varepsilon^2 \delta_x^2 v_i^n, & 1\leq i \leq M-1, 1\leq n \leq N, \\
\mathscr{A} v_i^n = -\delta_x^2 u_i^n, & 1\leq i \leq M-1, 1\leq n \leq N, \\
u_{i}^{0} = u_0(x_i), & 1\leq i \leq M-1.
\end{cases}
\end{equation}

\section{Theoretical analysis of the numerical scheme}
\label{sec:theoretical-analysis}

\subsection{Unique solvability}\label{sec:unique-solvability}
The unique solvability of the numerical scheme (\ref{eq.3.5}) can be established through the following framework. For arbitrary mesh functions $u,v \in \mathring{\mathcal{U}}_h$, we first introduce the essential inner products and associated norms:
\begin{align*}
&(u,v)=h\sum\limits_{i=1}^{M-1}u_{i}v_{i},\quad \left\|u\right\|=\sqrt{(u,u)}, \quad
\left(\delta_xu,\delta_xv\right)=\frac{1}{h}
\sum\limits_{i=1}^{M}(u_i-u_{i-1})(v_i-v_{i-1}), \\ &\left\|\delta_xu\right\|=\sqrt{(\delta_xu,\delta_xu)},\quad
\left(\delta_x^2u,\delta_x^2v\right)=
h\sum\limits_{i=1}^{M-1}\left(\delta_x^2u_{i}\right)\left(\delta_x^2v_{i}\right), \quad\left\|\delta_x^2u\right\|=\sqrt{\left(\delta_x^2u,\delta_x^2u\right)},\\
&\left(u,v\right)_{1,\mathcal{A}}=\left(\mathscr{A}u,-\delta_x^2v\right),\quad
\left\|\delta_xu\right\|_\mathcal{A}=\sqrt{\left(u,u\right)_{1,\mathcal{A}}},\quad
\left\|u\right\|_\infty=\max\limits_{1\leq i \leq M-1}\left|u_{i}\right|.
\end{align*}

The following lemma provides crucial inequalities for the analysis:

\begin{lemma}\label{Le.4.1}(\cite{Sun2012})
For any mesh function $u\in \mathring{\mathcal{U}}_h$, the following inequalities hold:
\begin{align*}
\frac{2}{3}\left\|\delta_x u\right\|^2 \leq \left\|\delta_x u\right\|_\mathcal{A}^2 \leq \left\|\delta_x u\right\|^2 \quad \text{and} \quad \frac{1}{3}\left\|u\right\|^2 \leq \left\|\mathscr{A}u\right\|^2 \leq \left\|u\right\|^2.
\end{align*}
\end{lemma}

Using these definitions, the numerical scheme (\ref{eq.3.5}) can be reformulated in operator notation as:
\begin{equation*}
\begin{cases}\displaystyle
\mathscr{A}\sum\limits_{k=1}^{n}B_{n-k}^{(n)}\triangledown_\tau u^k = \kappa\delta_x^2\left(u^{n}\right)^{\circ 3} - \kappa\delta_x^2 u^{n} + \kappa\varepsilon^2\delta_x^2 v^n, \\
\mathscr{A}v^n = -\delta_x^2 u^n,
\end{cases}
\end{equation*}
where the Hadamard power operation $(u^{n})^{\circ 3}$ denotes the element-wise cubic power of the solution vector $u^{n}$, defined recursively through the Hadamard product $\circ$ as $(u^{n})^{\circ p} = (u^{n})^{\circ (p-1)} \circ u^{n}$.

Now, we establish the unique solvability of the numerical scheme (\ref{eq.3.5}) through the following theorem:

\begin{theorem}\label{Th.4.2}
The nonlinear difference scheme (\ref{eq.3.5}) admits a unique solution provided the maximum time step satisfies
\begin{align*}
\tau_{\max} \leq \sqrt[\alpha]{\frac{(2-\alpha+2\rho_n)h^2}{12\kappa(1+\rho_n)\Gamma(3-\alpha)}}
\end{align*}
for all $n \geq 1$.
\end{theorem}

\begin{proof}
The initial value $u^{0}$ is directly determined by the initial condition (\ref{eq.1.1}). Proceeding by mathematical induction, assume the solutions ${u^{1}, u^{2}, \cdots, u^{n-1}}$ have been uniquely determined. To establish the unique solvability at time level $n$, it suffices to demonstrate that the corresponding homogeneous system
\begin{align*}
\begin{cases}
\mathscr{A} B_{0}^{(n)} u^n = \kappa \delta_x^2 \left(u^{n}\right)^{\circ 3} - \kappa \delta_x^2 u^{n} + \kappa \varepsilon^2 \delta_x^2 v^n, \\
\mathscr{A} v^n = -\delta_x^2 u^n,
\end{cases}
\end{align*}
possesses only the trivial solution.

Taking the inner product of the first equation with
$\mathscr{A} u^n$ yields
\begin{align*}
    B_{0}^{(n)} \left\|\mathscr{A} u^n\right\|^2 =
    -\kappa \left\|\delta_x(u^{n})^{\circ 2}\right\|_\mathcal{A}^2
     + \kappa \left\|\delta_x u^{n}\right\|_\mathcal{A}^2
      - \kappa \varepsilon^2 \left(v^n,u^n\right)_{1,\mathcal{A}}.
\end{align*}
Similarly, the inner product of the second equation
 with $\mathscr{A} v^n$ gives
\begin{align*}
    \left\|\mathscr{A} v^n\right\|^2 =
     \left(v^n,u^n\right)_{1,\mathcal{A}}.
\end{align*}

Combining these results leads to the inequality
\begin{align*}
    B_{0}^{(n)} \left\|\mathscr{A} u^n\right\|^2
     - \kappa \left\|\delta_x u^{n}\right\|_\mathcal{A}^2 = -\kappa \left\|\delta_x\left(u^{n}\right)^{\circ 2}\right\|_\mathcal{A}^2
     - \kappa \varepsilon^2 \left\|\mathscr{A} v^n\right\|^2 \leq 0.
\end{align*}

The non-positivity of the right-hand side implies $|u^{n}|^2 = 0$ when the condition
\begin{align*}
    B_{0}^{(n)} \left\|\mathscr{A} u^n\right\|^2
    - \kappa \left\|\delta_x u^{n}\right\|_\mathcal{A}^2
    \geq \left(\frac{B_{0}^{(n)}}{3}
    - \frac{4\kappa}{h^2}\right) \left\|u^{n}\right\|^2 \geq 0
\end{align*}
is satisfied by follows from definition (\ref{eq.2.6}) and Lemma \ref{Le.4.1}. This condition holds because
\begin{align*}
    B_0^{(n)} \geq c_0^{(n)} + \frac{\rho_n}{1+\rho_n} d_0^{(n)} = \frac{2-\alpha+2\rho_n}{(1+\rho_n)\Gamma(3-\alpha)\tau_n^\alpha} \geq \frac{12\kappa}{h^2}.
\end{align*}

The mathematical induction principle therefore guarantees the unique solvability of the scheme at all time levels, which completes the proof.
\end{proof}

\subsection{Discrete volume conservation}\label{sec: Discrete volume conservation property}
This subsection establishes the discrete volume conservation property (\ref{eq.1.4}) for the proposed numerical scheme (\ref{eq.3.5}). The conservation law plays a crucial role in maintaining the physical fidelity of the numerical solution.

\begin{theorem}\label{Th.4.3}
The numerical scheme (\ref{eq.3.5}) maintains volume
conservation in the discrete sense, satisfying
\begin{align*}
\int_\Omega u(x,t_n)\mathrm{d}x =
\int_\Omega u(x,t_{n-1})\mathrm{d}x \quad \text{for all} \; 1\leq n\leq N.
\end{align*}
\end{theorem}

\begin{proof}
The proof proceeds through several steps combining continuum analysis with discrete induction. First, consider the continuous equation (\ref{eq.1.1}). Integration over the domain $\Omega$ followed by application of Green's formula yields
\begin{align*}
\int_\Omega \partial_t^\alpha u(x,t) \;\mathrm{d}x
= \kappa \int_\Omega \Delta f(u) +
\varepsilon^2 \Delta v(x,t)\; \mathrm{d}x = 0,
\end{align*}
where the vanishing right-hand side results from boundary conditions.

Discretizing this conservation property using the scheme (\ref{eq.2.5}) gives
\begin{align*}
\int_\Omega \sum\limits_{k=1}^{n}B_{n-k}^{(n)} \left(u(x,t_k)-u(x,t_{k-1}) \right)\mathrm{d}x = 0
\end{align*}
for each $1\leq n\leq N$. The conservation property can be rigorously established through mathematical induction.

For the base case $n=1$, the temporal discretization scheme (\ref{eq.2.3}) combined with (\ref{eq.2.5}) produces
\begin{align*}
\int_\Omega c_{0}^{(1)} \left(u(x,t_1)-u(x,t_{0}) \right)\mathrm{d}x = 0,
\end{align*}
immediately implying $\int_\Omega u(x,t_1)\mathrm{d}x = \int_\Omega u(x,t_{0}) \mathrm{d}x$.

Assuming the conservation property holds for all time steps up to $N-1$, i.e.,
\begin{align*}
\int_\Omega u(x,t_n)\mathrm{d}x = \int_\Omega u(x,t_{n-1})\mathrm{d}x \quad \text{for} \; 1\leq n\leq N-1,
\end{align*}
we examine the $N$-th time step. The discrete equation becomes
\begin{align*}
\int_\Omega \sum_{k=1}^{N} B_{N-k}^{(N)} \left(u(x,t_k)-u(x,t_{k-1}) \right)\mathrm{d}x = \int_\Omega B_{0}^{(N)} \left(u(x,t_N)-u(x,t_{N-1}) \right)\mathrm{d}x = 0,
\end{align*}
where the simplification follows from the induction hypothesis. This directly yields
\begin{align*}
\int_\Omega u(x,t_N) \mathrm{d}x = \int_\Omega u(x,t_{N-1})\mathrm{d}x,
\end{align*}
completing the inductive step and the proof.
\end{proof}

\subsection{Energy dissipation law}\label{sec: Energy Dissipation Law}
The energy dissipation law of the numerical scheme are established through careful analysis of the operator $\mathcal{H} = \mathscr{A}^{-1}\delta_x^2$. This operator plays a fundamental role in the discrete energy analysis and possesses the following essential properties:
\begin{lemma}\label{Le.4.4}
For the operator $\mathcal{H} = \mathscr{A}^{-1}\delta_x^2$ acting on the space $\mathring{\mathscr{U}}_h$, the following properties hold:
\begin{itemize}
\item[(1)] The operator $\mathcal{H}$ is symmetric, satisfying $(\mathcal{H} u, v) = (u, \mathcal{H} v)$ for all $u, v \in \mathring{\mathscr{U}}_h$.
\item[(2)] The operator $\mathcal{H}$ is negative semidefinite, with $(\mathcal{H} u, u) \leq 0$ for any $u\in \mathring{\mathcal{U}}_h$.
\end{itemize}
\end{lemma}

\begin{proof}
According to the operator definition and inner product property, we derive:
\begin{align*}
\left(\mathcal{H} u, v\right) = \left(\mathcal{A}^{-1}\delta_x^2 u, v\right)
= \left(\delta_x^2 u, \mathcal{A}^{-1} v\right)
= \left(u, \delta_x^2 \left(\mathcal{A}^{-1} v\right)\right)
= \left(u, \mathcal{H} v\right).
\end{align*}

For the negative semidefinite property (2), Lemma \ref{Le.4.1} provides the necessary framework for the calculation:
\begin{align*}
	\left(\mathcal{H} u, u\right) = \left(\mathcal{A}^{-1}\delta_x^2 u, u\right) = \left(\delta_x^2 \mathcal{A}^{-1}u, u\right)
= -\left(\delta_x^2 \left(\mathcal{A}^{-1} u\right),
-\mathcal{A} \left(\mathcal{A}^{-1} u\right)\right)
= -\left\|\delta_x \left(\mathcal{A}^{-1} u\right)\right\|_\mathcal{A}^2 \leq 0.
\end{align*}
This completes the proof of both properties.
\end{proof}

Building upon equations (\ref{eq.2.7}), (\ref{eq.3.2}), and Lemma \ref{Le.4.4}, we derive the following discrete evolution equation:
\begin{equation}\label{eq.4.1}
\begin{aligned}
&\left( \theta c_{0}^{(n)}+ \frac{\rho_{n}}{1+\rho_{n}}d_{0}^{(n)} \right) \triangledown_\tau u^{n}
-\frac{\rho_{n}^2}{1+\rho_{n}}d_{0}^{(n)} \triangledown_\tau u^{n-1} + \sum\limits_{k=1}^{n} \widetilde{c}_{n-k}^{(n)}\triangledown_\tau u^{k} \\
= &\kappa \mathcal{H} \left(u^{n} \right)^{\circ 3}-\kappa \mathcal{H} u^{n} - \kappa \varepsilon^2 \mathcal{H}\left(\mathcal{H} u^n\right).
\end{aligned}
\end{equation}

To properly analyze the energy dissipation, we introduce an $\mathcal{H}^{-1}$-like inner product structure. For arbitrary mesh functions $u,v \in \mathring{\mathcal{U}}_h$, define
\begin{align*}
(u,v)_{-\mathcal{H}} := \left(u,\left(-\mathcal{H}\right)^{-1} v \right)
\end{align*}
with the induced norm $\left\|u\right\|_{-\mathcal{H}} :=
\sqrt{\left(u,u\right)_{-\mathcal{H}}}$. This structure leads to the definition of a modified discrete energy functional:
\begin{align}\label{eq.4.2}
\mathcal{E}^n = E^n + \frac{1}{\kappa}\left(\mathcal{G}\left
[\triangledown_\tau u^n\right],1\right)_{-\mathcal{H}} \quad \text{for} \; n \ge 1,
\end{align}
where $\mathcal{G}$ represents the non-negative functional from Theorem \ref{Th.2.6}. The discrete free energy $E^n$ corresponds to the discretization of the continuous energy functional (\ref{eq.1.6}) and is given by
\begin{align*}
E^n := \frac{\varepsilon^2}{2} \left\| \nabla_h u^{n} \right\|^2
 + \frac{1}{4} \left\| \left(u^{n}\right)^{\circ 2}-1 \right\|^2,
\end{align*}
with $\nabla_h$ denoting a discrete gradient operator satisfying $(\mathcal{H} u,v) = -(\nabla_h u,\nabla_h v)$.

The following theorem establishes the energy dissipation law of the numerical scheme under appropriate time step constraints:

\begin{theorem}\label{Th.4.5}
Under the conditions of Theorem \ref{Th.4.2} with the time step restriction
\begin{align*}
\tau_n \leq \sqrt[\alpha]{\frac{4\epsilon^2 q(\rho_n,\rho_{n+1},\alpha)}{\kappa\Gamma(3-\alpha)}} \quad \text{for} \quad 2 \leq n \leq N,
\end{align*}
the modified energy satisfies the dissipation property
\begin{align*}
\mathcal{E}^n \leq \mathcal{E}^{n-1}, \quad n=2,3,\cdots,N.
\end{align*}
\end{theorem}

\begin{proof}
First, taking the inner product of equation (\ref{eq.4.1}) with $(-\mathcal{H})^{-1}\triangledown_\tau u^n/\kappa$ yields
\begin{equation}\label{eq.4.3}
\begin{aligned}
&\frac{1}{\kappa} \Bigg( \left( \theta c_{0}^{(n)}+ \frac{\rho_{n}}{1+\rho_{n}}d_{0}^{(n)} \right) \triangledown_\tau u^{n} - \frac{\rho_{n}^2}{1+\rho_{n}}d_{0}^{(n)} \triangledown_\tau u^{n-1}
+ \sum\limits_{k=1}^{n} \widetilde{c}_{n-k}^{(n)} \triangledown_\tau u^{k}, \triangledown_\tau u^n \Bigg)_{-\mathcal{H}} \\
= &\left(u^{n}- \left(u^{n} \right)^{\circ 3}, \triangledown_\tau u^n \right) + \varepsilon^2 \left(\mathcal{H} u^n, \triangledown_\tau u^n \right).
\end{aligned}
\end{equation}

	Next, let's analyze the two sides of equation (\ref{eq.4.3}). Firstly, for its left side, Theorem \ref{Th.2.6} yields
	\begin{equation}\label{eq.4.4}
		\begin{aligned}\displaystyle
			&\frac{1}{\kappa} \Bigg{(} \left( \theta c_{0}^{(n)}+  \frac{\rho_{n}}{1+\rho_{n}}d_{0}^{(n)} \right) \triangledown_\tau u^{n}-\frac{\rho_{n}^2}{1+\rho_{n}}d_{0}^{(n)} \triangledown_\tau u^{n-1}  +\sum\limits_{k=1}^{n} \widetilde{c}_{n-k}^{(n)} \triangledown_\tau u^{k},\triangledown_\tau u^n \Bigg{)}_{-\mathcal{H}} \\
		\ge	&\frac{1}{\kappa}\left(\mathcal{G}
\left[\triangledown_\tau u^n\right],1\right)_{-\mathcal{H}} -\frac{1}{\kappa}\left(\mathcal{G}\left[\triangledown_\tau
 u^{n-1}\right],1\right)_{-H} +\frac{q\left(\rho_n,\rho_{n+1},\alpha\right)}{2\kappa \tau_{n}^\alpha \Gamma(3-\alpha)}\left\|\triangledown_\tau u^{n}\right\|_{-\mathcal{H}}^2.
		\end{aligned}
	\end{equation}

For the right-hand side terms, we employ the polynomial inequality
\begin{align*}
\left(\widetilde{a}-\widetilde{a}^3\right)\left(\widetilde{a}-\widetilde{b}\right) \leq \frac{1}{2}(\widetilde{a}-\widetilde{b})^2 - \frac{1}{4}\left[ \left(\widetilde{a}^2-1\right)^2-\left(\widetilde{b}^2-1\right)^2\right],
\end{align*}
which leads to
\begin{align*}
	\left(u^{n}- \left(u^{n} \right)^{\circ 3}, \triangledown_\tau u^n \right) \leq \frac{1}{2}\|\triangledown_\tau u^n\|^2 - \frac{1}{4} \left\| \left(u^{n}\right)^{\circ 2}-1 \right\|^2 + \frac{1}{4} \left\| \left(u^{n-1}\right)^{\circ 2}-1 \right\|^2.
\end{align*}

The second term on the right-hand side of equation (\ref{eq.4.3})
can be expressed using the identity
\begin{align*}
	\left(\mathcal{H} u^n, \triangledown_\tau u^n \right) &= -\frac{1}{2}\left( \nabla_h (\triangledown_\tau u^n + u^n + u^{n-1}), \nabla_h \triangledown_\tau u^n \right) \\
	&= \frac{1}{2} \left\| \nabla_h u^{n-1} \right\|^2 - \frac{1}{2} \left\| \nabla_h u^{n} \right\|^2 - \frac{1}{2} \left\| \nabla_h \triangledown_\tau u^n \right\|^2.
\end{align*}

Combining these estimates yields the bound for the right-hand side of equation (\ref{eq.4.3}) in the form
\begin{equation}\label{eq.4.5}
	\begin{aligned}
		&\left(u^{n}- \left(u^{n} \right)^{\circ 3}, \triangledown_\tau u^n \right) + \varepsilon^2 \left(\mathcal{H} u^n, \triangledown_\tau u^n \right) \\
		\leq &\frac{1}{2}\|\triangledown_\tau u^n\|^2 - \frac{1}{4} \left\| \left(u^{n}\right)^{\circ 2}-1 \right\|^2 + \frac{1}{4} \left\| \left(u^{n-1}\right)^{\circ 2}-1 \right\|^2 \\
		&- \frac{\varepsilon^2}{2} \left\| \nabla_h u^{n} \right\|^2 + \frac{\varepsilon^2}{2} \left\| \nabla_h u^{n-1} \right\|^2 - \frac{\varepsilon^2}{2} \left\| \nabla_h \triangledown_\tau u^n \right\|^2.
	\end{aligned}
\end{equation}
Applying the generalized Hölder inequality
\begin{align*}
	\|\triangledown_\tau u^n\|^2 \leq \varepsilon^2 \left\| \nabla_h \triangledown_\tau u^n \right\|^2 + \frac{1}{4\varepsilon^2}\|\triangledown_\tau u^{n} \|_{-\mathcal{H}}^2,
\end{align*}
and combining with (\ref{eq.4.2})-(\ref{eq.4.5}), we obtain the key estimate:
		\begin{align*}
			\mathcal{E}^n- \mathcal{E}^{n-1}
			\leq \left(\frac{1}{8\varepsilon^2}- \frac{q(\rho_n,\rho_{n+1},\alpha) }{2\kappa \tau_{n}^\alpha \Gamma(3-\alpha)}\right) \|\triangledown_\tau u^{n} \|_{-\mathcal{H}}^2 .
		\end{align*}
The desired result $\mathcal{E}^n \leq \mathcal{E}^{n-1}$ follows when the coefficient satisfies
\begin{align*}
	\frac{1}{8\varepsilon^2} - \frac{q(\rho_n,\rho_{n+1},\alpha)}{2\kappa \tau_{n}^\alpha \Gamma(3-\alpha)} \leq 0,
\end{align*}
which completes the proof.
\end{proof}

\begin{remark}\label{Re.2}
For the initial time step $n=1$, a similar analysis shows that under the condition $\tau_1 \leq \sqrt[\alpha]{\frac{8\epsilon^2}{\kappa\Gamma(2-\alpha)}}$, the energy satisfies $E^1 \leq E^0$.
As the fractional index $\alpha \to 1^-$, the discrete kernels $J_{n-k}^{(n)} \to 0$ for $1 \leq k \leq n$ according to Theorem \ref{Th.2.5} and equation (\ref{eq.2.8}). This asymptotic behavior demonstrates that the modified energy $\mathcal{E}^n$ converges to the original Ginzburg-Landau energy:
\begin{align*}
\mathcal{E}^n \to E^n \quad \text{as} \quad \alpha \to 1^-,
\end{align*}
showing the consistency of our discrete formulation with the classical case.
\end{remark}

\subsection{Convergence analysis}\label{sec: Convergence analysis}

This section establishes the convergence properties of the proposed numerical scheme through careful error analysis. We begin by introducing a modified Gr\"{o}nwall inequality that will be instrumental in our analysis.

\begin{lemma}\label{Le.4.6}
Let $\widetilde{C}>0$, and let $\left\{y_n\right\}$, $\left\{g_n\right\}$,
 and $\left\{\phi_n\right\}$ be non-negative sequences satisfying
\begin{align*}
\phi_{0}\leq \widetilde{C},\quad \phi_{n}\leq \widetilde{C}+\sum\limits_{j=0}^{n-1}y_{k}+\sum\limits_{j=0}^{n-1}g_{k}\phi_{k}, \quad n\geq 1.
\end{align*}
Then the following inequality holds:
\begin{align*}
\phi_n \leq \left(\widetilde{C}+\sum\limits_{j=0}^{n-1}y_{k} \right) \exp\left(\sum_{k=0}^{n-1}g_k\right), \quad n\geq 1.
\end{align*}
\end{lemma}

Let $l^2(\mathbb{Z})$ denote the space of square-summable bilateral sequences, i.e., sequences satisfying $\sum_{i=-\infty}^{\infty}|u_i|^2<\infty$. We establish the following property of the operator $\mathcal{A}^{-1}$:

\begin{lemma}\label{Le.4.7}
The inverse operator $\mathcal{A}^{-1}$ acting on $l^2(\mathbb{Z})$ has operator norm $\frac{3}{2}$, that is, $\left\|\mathcal{A}^{-1}\right\|=\frac{3}{2}$.
\end{lemma}

\begin{proof}
The operator $\mathcal{A}: l^2(\mathbb{Z}) \to l^2(\mathbb{Z})$ can be represented as a convolution operator with kernel $\left\{a_k\right\}_{k=-\infty}^{\infty}$ given by
\begin{align*}
a_k =
\begin{cases}
\frac{1}{12} & \text{if } k = -1, \\
\frac{5}{6} & \text{if } k = 0, \\
\frac{1}{12} & \text{if } k = 1, \\
0 & \text{otherwise}.
\end{cases}
\end{align*}
Thus, the action of $\mathcal{A}$ can be expressed
as $\left(\mathcal{A}u\right)_i =
\sum_{k=-\infty}^{\infty} a_k u_{i-k}$.
Through discrete-time Fourier transform analysis,
the symbol function of $\mathcal{A}$ is given by
$
m(\theta) = \sum\limits_{k=-\infty}^{\infty}a_k e^{-\mathrm{i}
 k \theta}, \;\theta \in [0, 2\pi).
$
Substituting the kernel coefficients and using Euler's formula yields
\begin{align*}
m(\theta) = \frac{1}{12} e^{\mathrm{i}\theta} + \frac{5}{6}
+ \frac{1}{12} e^{-\mathrm{i}\theta} = \frac{5 + \cos \theta}{6}.
\end{align*}
Since $\cos \theta \in [-1,1]$, we have $m(\theta) \in [{2}/{3},1]$. The positivity of $m(\theta)$ guarantees the invertibility of $\mathcal{A}$.

The symbol function of the inverse operator $\mathcal{A}^{-1}$ is given by
$
\frac{1}{m(\theta)} = \frac{6}{5 + \cos \theta},
$
from which we obtain the operator norm
\begin{align*}
\left\|\mathcal{A}^{-1}\right\| = \sup_{\theta \in [0, 2\pi)} \left|\frac{1}{m(\theta)}\right| = \sup_{\theta \in [0, 2\pi)}
\left|\frac{6}{5 + \cos \theta}\right| = \frac{3}{2}.
\end{align*}
\end{proof}

Based on these preliminary results, we now
establish the main convergence theorem.

\begin{theorem}\label{Th.4.8}
Let $u(x,t)$ be the exact solution of equation
(\ref{eq.1.1}) satisfying the regularity assumptions,
and let $f(u)$ satisfy the Lipschitz condition,
i.e.,
$\left|f(u^1) - f(u^2)\right| \leq L|u^1 - u^2|
$
for some constant $L>0$. Define the error term
\begin{align*}
e_{i}^n := u(x_i,t_n) - u_{i}^n, \quad 0 \leq n \leq N.
\end{align*}
Then, under the condition that $\tau_{\max} \leq \sqrt[\alpha]{\frac{2\epsilon^2}{\kappa L^2 \Gamma(2-\alpha)}}$ for $1 \leq n \leq N$, the following error estimate holds:
\begin{align*}
\left\| e^n\right\| \leq C_1 C\sqrt{b-a}
\begin{cases}
\left(\tau^{2-\alpha}+h^4\right), & n =1, \\
\left(\tau^{3-\alpha}+h^4\right), & 2\leq n \leq N.
\end{cases}
\end{align*}
\end{theorem}

\begin{proof}
First, by subtracting equation (\ref{eq.3.5}) from equation
 (\ref{eq.3.3}), we obtain the following error equation:
\begin{equation}\label{eq.4.7}
\begin{cases}
\mathscr{A} \sum\limits_{k=1}^{n}B_{n-k}^{(n)} \triangledown_\tau e_i^k = \kappa \delta_x^2 \left(f(u(x_i,t_n))-f(u_i^n)\right)
+ \kappa \varepsilon^2 \delta_x^2 \left(v(x_i,t_n)- v_i^n\right) + r_{i}^{n}, \\
\mathscr{A} \left(v(x_i,t_n)- v_i^n\right) = -\delta_x^2 e_i^n, \\
e_{i}^{0} = 0,
\end{cases}
\end{equation}
for $1\leq i \leq M-1, \; 1\leq n \leq N$.
For subsequent analysis, we reformulate equation (\ref{eq.4.7}) as
\begin{equation}\label{eq.4.8}
\begin{aligned}
&\left( \theta c_{0}^{(n)}+ \frac{\rho_{n}}{1+\rho_{n}}d_{0}^{(n)} \right) \triangledown_\tau e^{n} - \frac{\rho_{n}^2}{1+\rho_{n}}d_{0}^{(n)} \triangledown_\tau e^{n-1} + \widetilde{c}_{0}^{(n)} e^{n}
- \sum\limits_{k=1}^{n-1} \left( \widetilde{c}_{n-k-1}^{(n)} - \widetilde{c}_{n-k}^{(n)} \right) e^{k} \\
= &\kappa \mathcal{H} \left(f\left(u\left(x_i,t_n\right)\right)
-f\left(u_i^n\right)\right) - \kappa
\varepsilon^2 \mathcal{H} \left(\mathcal{H} e^n\right) + \mathscr{A}^{-1}r^n,
\end{aligned}
\end{equation}
valid for $n \geq 2$.
Taking the inner product of (\ref{eq.4.8}) with $e^n$ and applying the Lipschitz condition leads to
\begin{equation}\label{eq.4.9}
\begin{aligned}
&\Bigg(\left(\theta c_{0}^{(n)} + \frac{\rho_{n}}{1+\rho_{n}}d_{0}^{(n)}\right) \triangledown_\tau e^{n} - \frac{\rho_{n}^2}{1+\rho_{n}}d_{0}^{(n)} \triangledown_\tau e^{n-1}
+ \widetilde{c}_{0}^{(n)} e^{n} - \sum_{k=1}^{n-1} \left(\widetilde{c}_{n-k-1}^{(n)} - \widetilde{c}_{n-k}^{(n)}\right) e^{k}, e^n\Bigg) \\
\leq &\kappa L|\left(\mathcal{H} e^n,e^n\right)| - \kappa \varepsilon^2
\left\|\mathcal{H} e^n\right\|^2 + \left(\mathscr{A}^{-1}r^n, e^n\right).
\end{aligned}
\end{equation}

The right-hand side of (\ref{eq.4.9}) is estimated through careful application of the Cauchy-Schwarz inequality and Young's inequality:
	\begin{align*}
		& \kappa L|\left(\mathcal{H}e^n,e^n \right)|-\kappa \varepsilon^2 \|\mathcal{H} e^n \|^2
		+\left(\mathscr{A}^{-1}r^n, e^n\right)
		\leq  \kappa L\left(\frac{\varepsilon^2}{L}\|\mathcal{H} e^n\|^2+\frac{L} {4\varepsilon^2 }\|e^n\|^2 \right) \\
		&-\kappa \varepsilon^2 \|\mathcal{H} e^n \|^2+\left( \frac{\widetilde{c}_{n-1}^{(n)} }{2}+ \frac{ d_{1}^{(n)} }	{2\rho_{n}\left(1+\rho_{n} \right)}\right)\|e^n\|^2+\frac{\|\mathscr{A}^{-1}\|^2} {2\widetilde{c}_{n-1}^{(n)}+ \frac{2 d_{1}^{(n)} }	{\rho_{n}\left(1+\rho_{n} \right)} }\|r^n\|^2\\
		\leq&\left(\frac{\kappa L^2} {4\varepsilon^2 }+ \frac{\widetilde{c}_{n-1}^{(n)} }{2}+ \frac{ d_{1}^{(n)} }	{2\rho_{n}\left(1+\rho_{n} \right)}\right)\|e^n\|^2+\frac{9} {8\widetilde{c}_{n-1}^{(n)}+ \frac{8 d_{1}^{(n)} }	{\rho_{n}\left(1+\rho_{n} \right)} }\|r^n\|^2.
	\end{align*}
Similarly, the left-hand side of (\ref{eq.4.9}) can be estimated as:
\begin{align*}
		&\Bigg{(} \left( \theta c_{0}^{(n)}+  \frac{\rho_{n}}{1+\rho_{n}}d_{0}^{(n)} \right) \triangledown_\tau e^{n} -\frac{\rho_{n}^2}  {1+\rho_{n}}d_{0}^{(n)} \triangledown_\tau e^{n-1}  +\widetilde{c}_{0}^{(n)}  e^{n} -\sum\limits_{k=1}^{n-1} \left( \widetilde{c}_{n-k-1}^{(n)} -\widetilde{c}_{n-k}^{(n)} \right)  e^{k} ,e^n\Bigg{)}\\
		\ge &\frac{\theta c_{0}^{(n)}+  \frac{\rho_{n}}{1+\rho_{n}}d_{0}^{(n)}} {2}\left(\|e^n \|^2 -\|e^{n-1} \|^2\right)-\frac{\rho_{n}^2 d_{0}^{(n)}}{1+\rho_{n}}\left(\frac{1}{2\rho_{n}} \|e^n \|^2 +\rho_{n}\|e^{n-1} \|^2 +\rho_{n}\|e^{n-2} \|^2\right) \\ &+ \widetilde{c}_{0}^{(n)} \|e^n \|^2 -\sum\limits_{k=1}^{n-1}\frac{\widetilde{c}_{n-k-1}^{(n)} -\widetilde{c}_{n-k}^{(n)} }{2}\left(\|e^n \|^2+ \|e^k \|^2 \right) \\
		=&\left(\frac{\theta} {2}c_{0}^{(n)}+\frac{1 }{2}\widetilde{c}_{0}^{(n)} +\frac{1}{2}\widetilde{c}_{n-1}^{(n)}\right) \|e^n \|^2 -\left(\frac{\theta c_{0}^{(n)}+  \frac{\rho_{n}}{1+\rho_{n}}d_{0}^{(n)}} {2} +\frac{\rho_{n}^3 d_{0}^{(n)}}{1+\rho_{n}}\right)\|e^{n-1} \|^2\\
		&-\frac{\rho_{n}^3 d_{0}^{(n)}}{1+\rho_{n}} \|e^{n-2} \|^2
		-\sum\limits_{k=1}^{n-1}\frac{\widetilde{c}_{n-k-1}^{(n)} -\widetilde{c}_{n-k}^{(n)} }{2}\|e^k \|^2.
	\end{align*}
Combining these estimates leads to the following key recursive inequality
\begin{equation}\label{eq.4.10}
\begin{aligned}
\left(c_{0}^{(n)} - \frac{\kappa L^2}{2\varepsilon^2}\right)
\left\|e^n\right\|^2 \leq &\frac{9}{4\widetilde{c}_{n-1}^{(n)} + \frac{4 d_{1}^{(n)}}{\rho_{n}(1+\rho_{n})}}\left\|r^n\right\|^2
+ \sum\limits_{k=1}^{n-1}G^k\left\|e^k\right\|^2,
\end{aligned}
\end{equation}
where the coefficients $G^k$ are defined as
\begin{align*}
G^k =
\begin{cases}
\widetilde{c}_{n-k-1}^{(n)} - \widetilde{c}_{n-k}^{(n)},
 & k =1,2,\cdots,n-3, \; n\geq 4, \\
\widetilde{c}_{1}^{(n)} - \widetilde{c}_{2}^{(n)} + \frac{2\rho_{n}^3 d_{0}^{(n)}}{1+\rho_{n}}, & k=n-2, \; n \geq 3, \\
\widetilde{c}_{0}^{(n)} - \widetilde{c}_{1}^{(n)} + \theta c_{0}^{(n)} + \frac{\rho_{n}}{1+\rho_{n}}d_{0}^{(n)}
+ \frac{2\rho_{n}^3 d_{0}^{(n)}}{1+\rho_{n}}, & k=n-1, \; n\geq 2.
\end{cases}
\end{align*}
The condition $\tau_{\max} \leq \sqrt[\alpha]{\frac{2\epsilon^2}{\kappa L^2 \Gamma(2-\alpha)}}$ ensures that
$
    c_{0}^{(n)} = \frac{1}{\tau_{n}^\alpha \Gamma(2-\alpha)} \geq \frac{\kappa L^2}{2\varepsilon^2}.
$
From (\ref{eq.3.4}), the truncation error satisfies
\begin{equation}\label{eq.4.11}
\left\|r^n\right\|^2 \leq (b-a)C^2
\begin{cases}
\left(\tau^{2-\alpha}+h^4\right)^2, & n =1, \\
\left(\tau^{3-\alpha}+h^4\right)^2, & 2\leq n \leq N.
\end{cases}
\end{equation}

Applying Lemma \ref{Le.4.6} to (\ref{eq.4.10}) and incorporating (\ref{eq.4.11}) establishes the desired error estimate:
\begin{align*}
    \left\| e^n \right\| \leq C_1 C \sqrt{b-a}
    \begin{cases}
        \left(\tau^{2-\alpha} + h^4\right), & n = 1, \\
        \left(\tau^{3-\alpha} + h^4\right), & 2 \leq n \leq N,
    \end{cases}
\end{align*}
where $C_1 = \sqrt{\frac{9}{\left(c_{0}^{(n)} - \frac{\kappa L^2}{2\varepsilon^2}\right) \left(4\widetilde{c}_{n-1}^{(n)} + \frac{4 d_{1}^{(n)}}{\rho_{n}(1+\rho_{n})}\right)}} \exp\left(\sum\limits_{k=1}^{n-1} \frac{G^k}{c_{0}^{(n)} - \frac{\kappa L^2}{2\varepsilon^2}}\right)$. This ends all the proofs.
\end{proof}

\section{Numerical examples}\label{sec: Numerical examples}
This section presents detailed numerical implementations and validation of the proposed methods. The examples demonstrate the effectiveness of the non-uniform temporal grid and verify the theoretical results established in previous sections.

\subsection{Implementation details}

Following Theorem \ref{Th.2.6}, the time-step ratios $\rho_{k}$ must satisfy $1 \leq \rho_{k} \leq \rho^*(\alpha)$ for $k\ge 2$. To meet this requirement, we construct a special non-uniform grid defined by
\begin{align*}
\tau_k=t_{k}-t_{k-1}=\frac{(2k+1)^3}{N(N+2)(2N^2+4N+3)}T, \quad k=1,2,\cdots,N,
\end{align*}
which guarantees the condition:
\begin{align*}
1 < \rho_{k} = \left(\frac{2k+1}{2k-1}\right)^3 \leq \left(\frac{5}{3}\right)^3 \approx 4.6296296 < 4.7476114 \approx \overline{\rho} < \rho^*(\alpha).
\end{align*}
 These carefully designed meshes will be used throughout our numerical experiments for solving the TFCH equation (\ref{eq.1.1}).

 The nonlinear fully discrete scheme (\ref{eq.3.5}) is solved using a simple iteration method with termination error tolerance $10^{-10}$.

\subsection{Convergence tests}

To validate the convergence order of our constructed numerical differential formula,
the following computational experiment is conducted.

\begin{example}\label{ex.5.1}
The convergence properties of the discrete Caputo derivative using the L2 approximation formula are investigated through numerical verification. The test function
\begin{align*}
w(t) = t^{3+\alpha},\quad t\in [0,1],
\end{align*}
is employed, which possesses sufficient regularity with $w(t) \in C^3[0, 1]$. The exact Caputo derivative of this function is analytically known as
\begin{align*}
\partial^\alpha_t w(t)=\frac{\Gamma(4+\alpha)}{\Gamma(4)}t^3.
\end{align*}
\end{example}

The numerical accuracy is measured by the maximum norm error
\begin{align*}
    e(N) := \max_{1\leq n \leq N} \|w(t_n) - w^n\|_\infty,
\end{align*}
with the corresponding convergence order computed via
\begin{align*}
    \text{Order} := \log_2\left(\frac{e(N)}{e(2N)}\right).
\end{align*}

The numerical scheme (\ref{eq.2.5}) is implemented on the proposed nonuniform mesh for fractional orders $\alpha =0.3, 0.5, 0.7, 0.9$. Table~\ref{table.5.1} presents the detailed numerical results, including the errors and computed convergence orders for various temporal discretizations.

\begin{table}[htbp]
\caption{Convergence results for L2 approximation of Caputo derivative}\label{table.5.1}
\begin{center}
\renewcommand{\arraystretch}{1.1}
\begin{tabular}{c|cc|cc|cc|cc}
&\multicolumn{2}{c|}{$\alpha=0.3$}
&\multicolumn{2}{c|}{$\alpha=0.5$}
&\multicolumn{2}{c|}{$\alpha=0.7$}
&\multicolumn{2}{c}{$\alpha=0.9$} \\
\cline{2-3} \cline{4-5} \cline{6-7}\cline{8-9}
$N$&$e(N)$ & Order &$e(N)$ & Order &$e(N)$ & Order &$e(N)$ & Order\\
\hline
250 &3.89e-06 &-- &2.25e-05 &-- &9.30e-05 &-- &3.19e-04 &-- \\
500 &5.72e-07 &2.77 &3.98e-06 &2.50 &1.91e-05 &2.28 &7.58e-05 &2.07\\
1000 &8.36e-08 &2.77 &6.99e-07 &2.51 &3.91e-06 &2.29 &1.78e-05 &2.09\\
2000 &1.22e-08 &2.77 &1.23e-07 &2.51 &7.96e-07 &2.30 &4.18e-06 &2.09\\
4000 &1.80e-09 &2.77 &2.15e-08 &2.51 &1.62e-07 &2.30 &9.78e-07 &2.10\\
\hline
\end{tabular}
\end{center}
\end{table}
The numerical results demonstrate that the proposed mesh achieves the theoretical convergence order of $3-\alpha$ for all considered fractional orders. These findings provide strong evidence for two key conclusions: (1) the nonuniform mesh effectively handles the weak singularity of fractional operators at the initial time point, and (2) the numerical results validate the theoretical predictions established in Theorem~\ref{Th.2.1}.

\subsection{TFCH equation simulations}

The proposed numerical method is applied to the TFCH equation with various parameter configurations to examine its performance characteristics.

\begin{example}\label{ex.5.2}
 The time-fractional Cahn-Hilliard equation under consideration takes the form:
\begin{align*}
    \partial^\alpha_t u(x,t) = \kappa \Delta f(u) - \kappa \varepsilon^2 \Delta^2 u(x,t), \quad x \in (0,1),\; t \in (0,1],
\end{align*}
with the nonlinear term given by $f(u) = u^3 - u$. To assess the convergence properties of the numerical scheme, the following smooth initial condition is employed:
\begin{align*}
    u(x,0) = x^4(1-x)^4.
\end{align*}
\end{example}

In the absence of an analytical solution for this problem, the reference solution $U^{n}_{ref}$ is generated numerically using a refined discretization with $M = 60$ spatial grid points and $N_0 = 200$ temporal steps. The temporal convergence order at each time level $t_n$ is computed via the formula:
\begin{align*}
\mathrm{Order} := \frac{\log(e(N_1)/e(N_2))}{\log(N_2/N_1)},\quad \text{where}\;\;
e(N) := \left\|U^N-U^{N_0}_{ref}\right\|_\infty.
\end{align*}

The convergence analysis is performed with fixed parameters $\varepsilon = 0.1$, $\kappa = 0.01$, and final simulation time $T = 1$. Following the validation approach described in Example~\ref{ex.5.1}, Table~\ref{table.5.2} systematically presents the numerical errors and corresponding convergence orders for fractional orders $\alpha=0.3, 0.5, 0.7, 0.9$. The tabulated results demonstrate the consistent convergence behavior of the numerical method across different fractional orders and parameter settings.

\begin{table}[htbp]
\caption{Numerical errors and temporal convergence orders for the TFCH equation}
\label{table.5.2}
\centering
\renewcommand{\arraystretch}{1.1}
\begin{tabular}{c|cc|cc|cc|cc}
&\multicolumn{2}{c|}{$\alpha=0.3$}
&\multicolumn{2}{c|}{$\alpha=0.5$}
&\multicolumn{2}{c|}{$\alpha=0.7$}
&\multicolumn{2}{c}{$\alpha=0.9$} \\
\cline{2-3} \cline{4-5} \cline{6-7}\cline{8-9}
$N$&e(N) & Order &e(N) & Order &e(N) & Order &e(N) & Order\\
\hline
15 &6.09e-07 &-- &1.77e-06 &-- &2.50e-06 &-- &1.30e-06 &-- \\
18 &3.66e-07 &2.796 &1.12e-06&2.522 &1.62e-06&2.369 &8.84e-07&2.130\\
21 &2.40e-07 &2.738&7.54e-07&2.547&1.13e-06&2.382 &6.36e-07&2.139\\
24 &1.68e-07 &2.677&5.40e-07&2.498&8.18e-07&2.387 &7.76e-07&2.160\\
\hline
\end{tabular}
\end{table}

\subsection{Energy behavior analysis}

The energy properties of the numerical solutions are systematically investigated through the following computational experiments.

\begin{example}\label{ex.5.3}
Building upon the configuration specified in Example~\ref{ex.5.2}, numerical experiments are performed with fixed parameters $\varepsilon = 0.1$ and $\kappa = 0.01$, employing a computational mesh with $M = 60$ spatial nodes and $N = 200$ temporal steps.
\end{example}

The energy-preserving properties of the stabilized scheme (\ref{eq.3.5}) are examined by monitoring the evolution of both the discrete compatible energy $\mathcal{E}^n$ and the discrete free energy $E^n$ across successive time levels.
Computations are carried out for fractional orders $\alpha=0.2, 0.4, 0.6, 0.8$, with the results presented in Figure~\ref{Fig.5.1}. The numerical data clearly demonstrate that the discrete energy decay property is maintained throughout all simulations. Specifically, the computed solutions exhibit monotonic non-increasing energy profiles for each considered value of $\alpha$, confirming that the numerical scheme properly preserves the essential energy characteristics of the continuous system under the given initial conditions.
\begin{figure}[htbp]
    \centering
    \begin{subfigure}[b]{0.48\textwidth}
        \includegraphics[width=\textwidth]{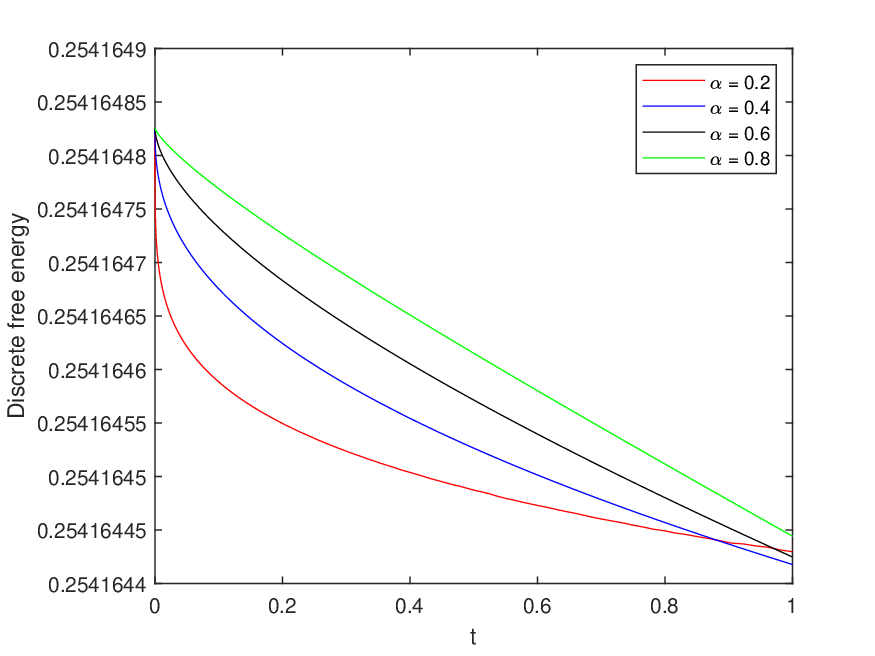} 
        \caption{Discrete free energy $E^n$ evolution}
        \label{fig:image1}
    \end{subfigure}
    \hfill 
    \begin{subfigure}[b]{0.48\textwidth}
        \includegraphics[width=\textwidth]{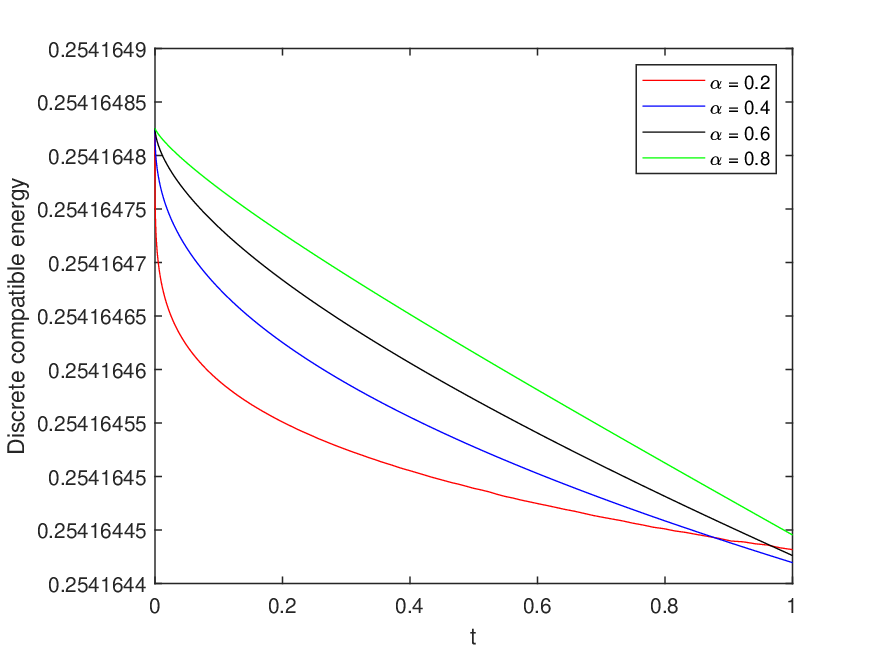} 
        \caption{Discrete compatible energy $\mathcal{E}^n$ evolution}
        \label{fig:image2}
    \end{subfigure}
    \caption{The discrete compatible energy $\mathcal{E}^n$ and the discrete free energy $E^n$ for different fractional orders $\alpha$}
    \label{Fig.5.1}
\end{figure}

\subsection{Mass conservation verification}
The mass conservation property is verified by systematic
numerical investigations, as presented in the following analysis.

\begin{example}\label{ex.5.4}
Continuing the numerical framework established in Example~\ref{ex.5.2}, the conservation of mass is examined quantitatively. The total mass at time level $t_n$ is numerically approximated through the second-order accurate composite trapezoidal rule:
\begin{align*}
\int_0^1u(x,t_n) \mathrm{d}x \approx h\left\{\frac{1}{2}\left(u_0^n+u_M^n\right)+\sum_{i=1}^{M-1}u_i^n\right\}.
\end{align*}
\end{example}

Building upon the theoretical foundation provided in Theorem~\ref{Th.4.3}, comprehensive numerical experiments are conducted for fractional orders $\alpha =0.2, 0.4, 0.6, 0.8$. The stabilized numerical scheme (\ref{eq.3.5}) is implemented with systematically varied parameters $M$, $N$, and $\varepsilon$. The resulting mass conservation properties are visualized in Figure~\ref{Fig5}, which demonstrates excellent preservation of total mass across all parameter combinations and confirms the theoretical predictions.

\begin{figure}[htbp]
    \centering
    \begin{subfigure}[b]{0.48\textwidth}
        \includegraphics[width=\textwidth]{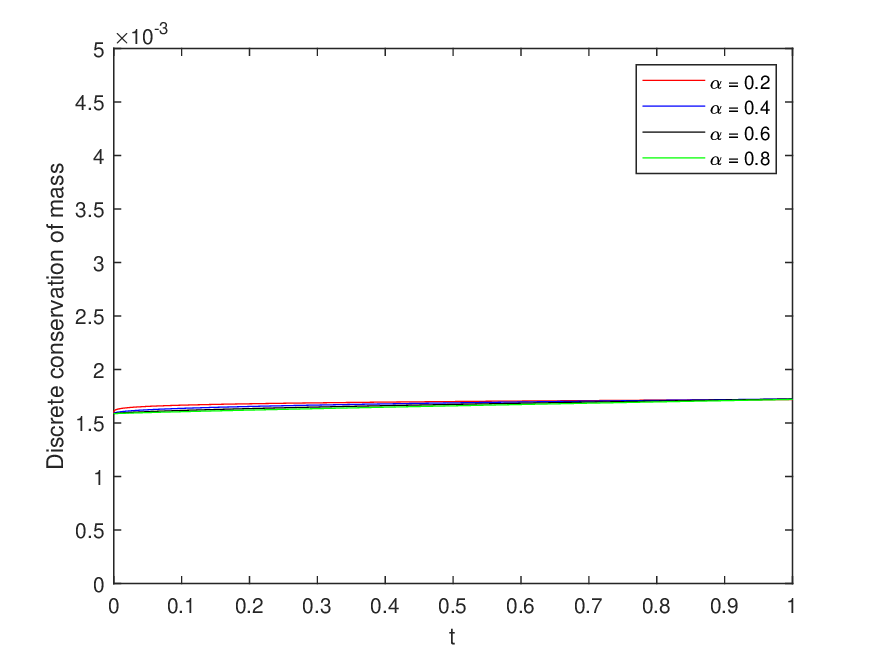}
        \caption{$\kappa=0.01$, $\varepsilon=0.1$, $M=60$, $N=200$}
    \end{subfigure}
    \hfill
    \begin{subfigure}[b]{0.48\textwidth}
        \includegraphics[width=\textwidth]{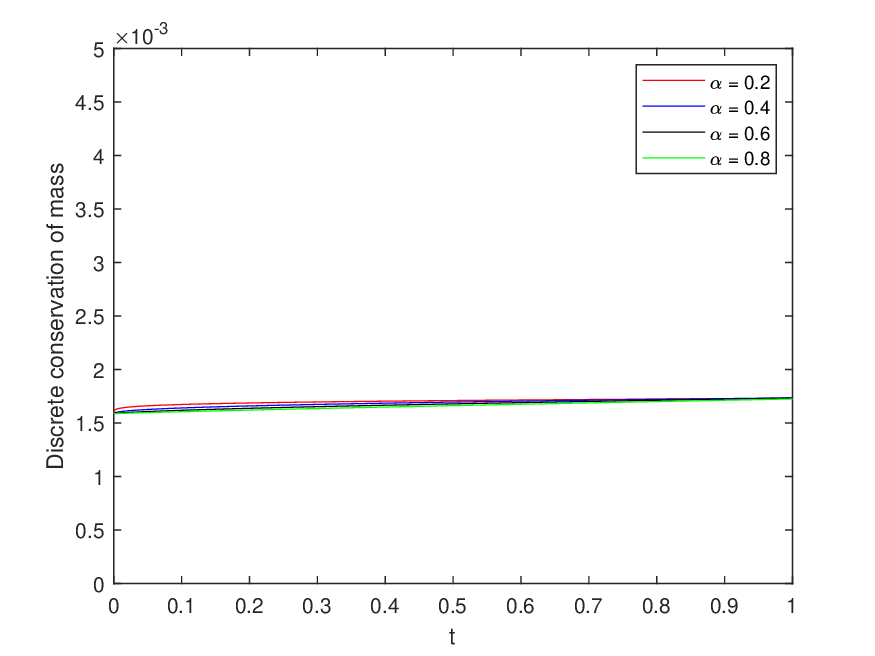}
        \caption{$\kappa=0.01$, $\varepsilon=0.1$, $M=100$, $N=200$}
    \end{subfigure}

    \vspace{0.5cm} 

    \begin{subfigure}[b]{0.48\textwidth}
        \includegraphics[width=\textwidth]{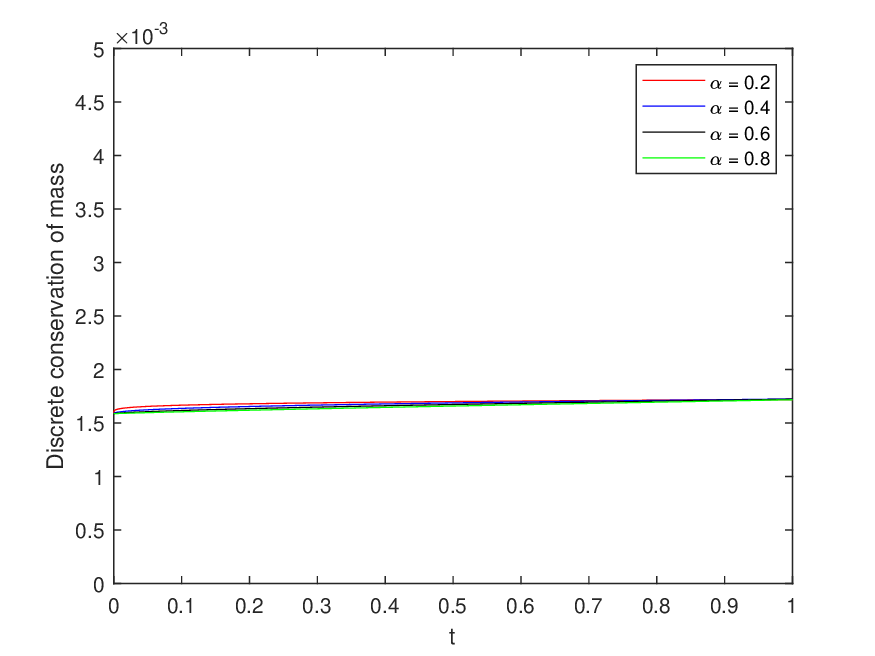}
        \caption{$\kappa=0.01$, $\varepsilon=0.1$, $M=60$, $N=100$}
    \end{subfigure}
    \hfill
    \begin{subfigure}[b]{0.48\textwidth}
        \includegraphics[width=\textwidth]{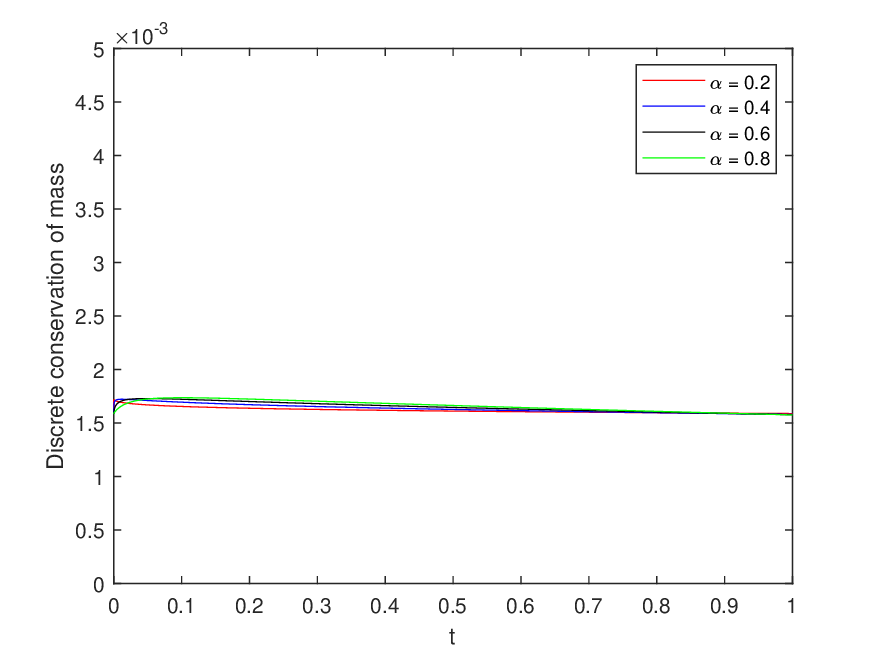}
        \caption{$\kappa=0.01$, $\varepsilon=0.5$, $M=60$, $N=200$}
    \end{subfigure}

    \caption{Comparison of total mass evolution for different parameters}
    \label{Fig5}
\end{figure}

\subsection{Exact solution verification}
Finally, we test our method against a problem with
known exact solution.

\begin{example}\label{ex.5.5}
Consider the TFCH equation with source term:
\begin{align*}
\partial^\alpha_t u(x,t) = \kappa \Delta f(u)-\kappa
\varepsilon^2 \Delta^2 u(x,t) +g(x,t), \quad x\in(0,1),\; t\in(0,1],
\end{align*}
where $f(u)=u^3-u$. The exact solution $u(x,t)=x^4 (1-x)^4
 t^{3+\alpha}$ determines the initial-boundary conditions
 and source term $g(x,t)$, given by
\begin{align*}
g(x,t)=&\frac{\Gamma(4+\alpha)}{\Gamma(4)}x^4 (1-x)^4t^3 -\kappa\left[ 132x^{10}(1-x)^{12} -288x^{11}(1-x)^{11} +132x^{12}(1-x)^{10}
\right]t^{9+3\alpha} \\
& +\kappa \varepsilon^2 \left[24(1-x)^{4}-384x (1-x)^3
+864x^2 (1-x)^2-384x^3 (1-x)+24x^4 \right]t^{3+\alpha} \\
&+\kappa\left[12x^{2}(1-x)^{4} -32x^{3}(1-x)^{3}
+12x^{4}(1-x)^{2} \right]t^{3+\alpha}.
\end{align*}
\end{example}

Fixed the constant $\kappa=0.01$ and $\varepsilon = 0.1$, Figure~\ref{Fig9} presents a systematic comparison between the exact solution $u(x,1)$ and the corresponding numerical approximations $u^N$ computed with different spatial discretization parameters $N =100,125,150,175,200$ for fractional orders $\alpha = 0.1, 0.3, 0.6, 0.9$. The numerical results demonstrate excellent agreement with the analytical solution across all parameter combinations, validating the accuracy of the proposed method.
\begin{figure}[htbp]
    \centering
    \begin{subfigure}[b]{0.48\textwidth}
        \includegraphics[width=\textwidth]{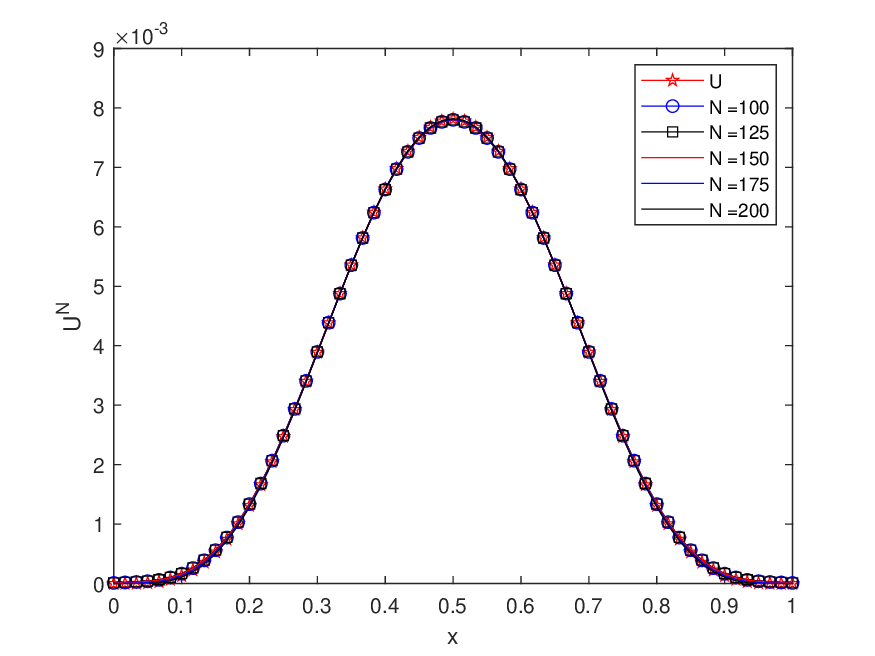}
        \caption{$\alpha=0.1$}
    \end{subfigure}
    \hfill
    \begin{subfigure}[b]{0.48\textwidth}
        \includegraphics[width=\textwidth]{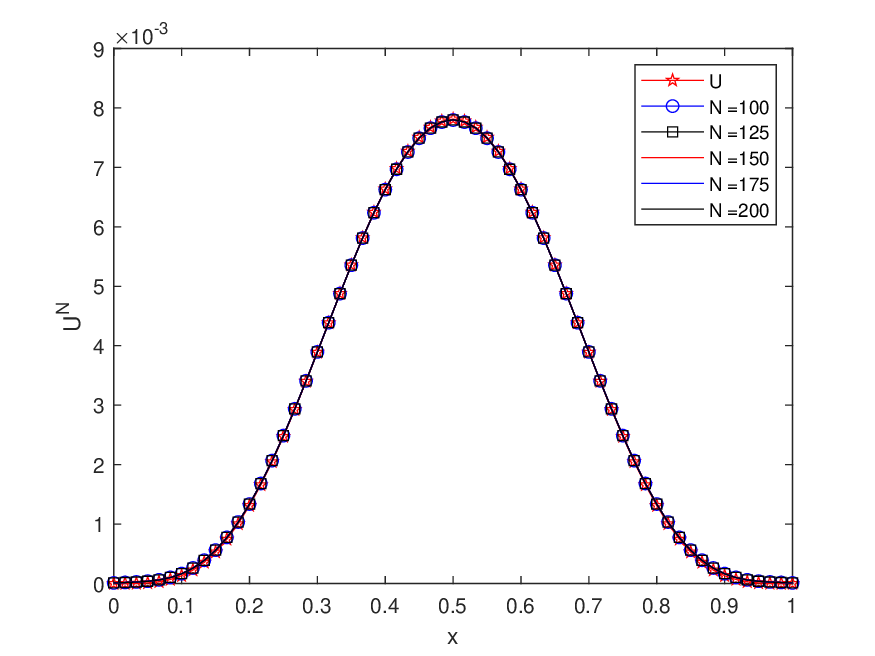}
        \caption{$\alpha=0.3$}
    \end{subfigure}

    \vspace{0.5cm} 

    \begin{subfigure}[b]{0.48\textwidth}
        \includegraphics[width=\textwidth]{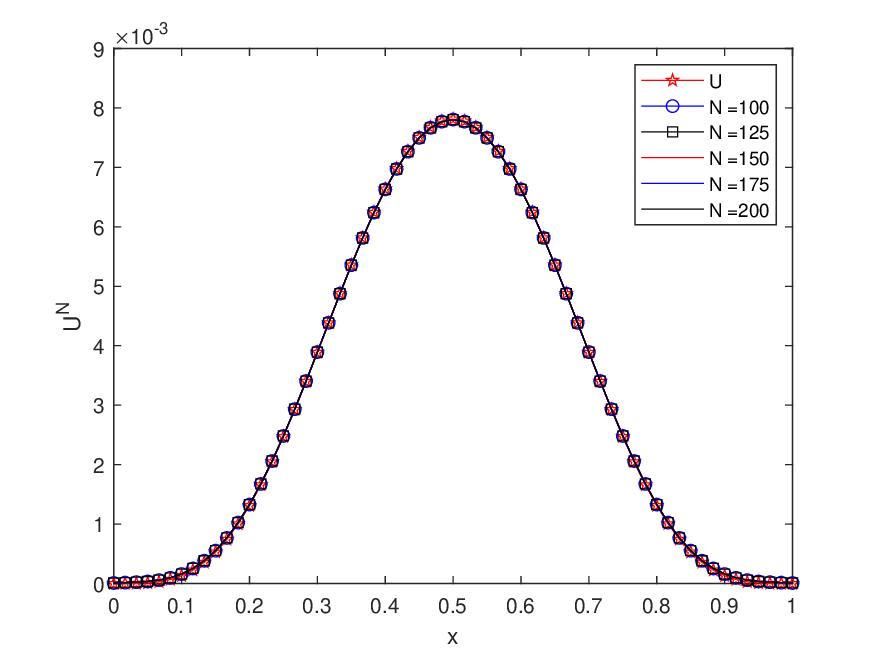}
        \caption{$\alpha=0.6$}
    \end{subfigure}
    \hfill
    \begin{subfigure}[b]{0.48\textwidth}
        \includegraphics[width=\textwidth]{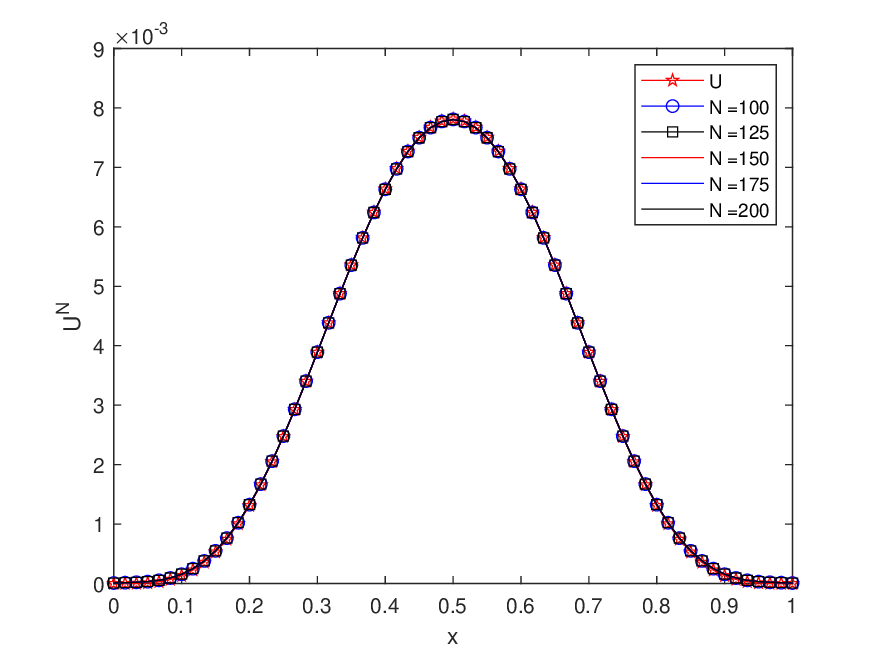}
        \caption{$\alpha=0.9$}
    \end{subfigure}

    \caption{Comparison of the exact and numerical solutions for
    different $\alpha$ values}
    \label{Fig9}
\end{figure}

The above numerical experiments provide rigorous verification of the theoretical analysis, establishing the effectiveness of the proposed numerical method for solving the time-fractional Cahn--Hilliard equation. As demonstrated in the computational results, the specially designed nonuniform temporal grid with step ratios satisfying $1 \leq \tau_k/\tau_{k-1} \leq \rho^*(\alpha)$, where $\rho^*(\alpha) > \overline{\rho} \approx 4.7476114$, successfully resolves the initial singularity while achieving the theoretically predicted convergence rates. Specifically, the numerical solutions exhibit the expected temporal convergence order of $3-\alpha$, confirming the theoretical estimates derived in Section~\ref{sec: Convergence analysis}.
Moreover, the proposed scheme maintains two fundamental physical properties of the continuous model: (1) exact volume conservation, and (2) energy dissipation throughout the simulation. These numerical observations, obtained under various parameter configurations, demonstrate the robustness of the method while preserving the essential features of the original physical system. The agreement between theoretical predictions and computational results across multiple test cases provides strong evidence for the reliability of the proposed approach.

\section{Concluding remarks}\label{sec: Concluding remarks}

This work has presented a high-order numerical scheme for the time-fractional Cahn-Hilliard equation, combining a $3-\alpha$-order temporal discretization on nonuniform grids with fourth-order spatial accuracy. The proposed method features several key theoretical advancements: an improved time-step ratio condition for the L2-type Caputo derivative approximation, rigorous proofs of unique solvability and discrete mass conservation, and establishment of energy decay properties. Numerical experiments comprehensively validate these theoretical results, demonstrating optimal convergence rates across different fractional orders while maintaining essential physical properties. The specially designed nonuniform temporal grids effectively handle the initial singularity, with the numerical results confirming both the high-order accuracy and stability of the proposed scheme.

The developed framework provides a solid foundation for future research on fractional phase-field models. Current efforts focus on extending this approach to more complex scenarios, including higher-dimensional problems, adaptive time-stepping strategies, and coupled systems involving multiple fractional operators. The combination of rigorous analysis and careful numerical implementation presented here offers a reliable template for solving challenging nonlinear fractional partial differential equations while preserving their fundamental mathematical and physical properties.

\bibliographystyle{siamplain}

\end{document}